\documentclass[reqno]{amsart}
\usepackage{amssymb,amsmath,hyperref}
\usepackage{amsrefs, float}
\usepackage[foot]{amsaddr}
\usepackage{bbold,stackrel, multicol}
 
\newtheorem{thm}{Theorem}

\newtheorem{cor}[thm]{Corollary}
\newtheorem{defi}[thm]{Definition}

\newtheorem{rem}[thm]{Remark}
\newtheorem{nota}[thm]{Notation}

\newtheorem{princ}[thm]{Principle}

\newtheorem{ack}[thm]{Acknowledgement}

\newtheorem*{tempo*}{Template}
\newtheorem{theorem}[thm]{Theorem}
\newtheorem{problem}[thm]{Problem}

\newtheorem{lemma}[thm]{Lemma}
\newtheorem{definition}[thm]{Definition}

\newtheorem{remark}[thm]{Remark}

\newcommand\be{\begin{equation}}
\newcommand\ee{\end{equation}} 

\usepackage{amsmath,amsfonts} 

\usepackage[applemac]{inputenc}

\def\bdefi{\begin{defi}\rm}
\def\edefi{\end{defi}}
\def\bnota{\begin{nota}\rm}
\def\enota{\end{nota}}

\def\blambda{\pmb{\lambda}}

\def\ZFC{\textup{\textsf{ZFC}}}

\def\({\textup{(}}
\def\){\textup{)}}

\def\bye{\end{document}}
\def\P{\textup{\textsf{P}}}
\def\N{{\mathbb  N}}
\def\Q{{\mathbb  Q}}
\def\R{{\mathbb  R}}

\def\SS{\textup{\textsf{S}}}

\def\di{\rightarrow}

\def\asa{\leftrightarrow}

\def\BV{\textup{\textsf{BV}}}

\def\eps{\varepsilon}

\def\usco{\textup{\textsf{usco}}}

\def\osc{\textup{\textsf{osc}}}

\def\CH{\textup{\textsf{CH}}}

\def\INT{\textup{\textsf{int}}}
\def\CK{\textup{\textsf{CK}}}

\newcommand{\F}{{\bf F}}

\newcommand{\Pa}{{\bf P}}
\newcommand{\suc}{\mathsf {suc}}
\newcommand{\pd}{\mathsf {pred}}
\newcommand{\case}{\mathsf {case}}
\newcommand{\Mo}{{\bf HC}}
\def\Fix{\textup{\textsf{Fix}}}
\newcommand{\Fixa}{{\rm Fix}}

\usepackage{mathrsfs}  

\usepackage{graphicx}
\usepackage{tikz}
\usetikzlibrary{matrix, shapes.misc}
\usepackage{comment,tikz-cd}

\numberwithin{equation}{section}
\numberwithin{thm}{section}

\begin{document}
\title{On some computational properties of open sets}
\author{Dag Normann}
\address{Department of Mathematics, The University 
of Oslo, P.O. Box 1053, Blindern N-0316 Oslo, Norway}
\email{dnormann@math.uio.no}
\author{Sam Sanders}

\address{Department of Philosophy II, RUB Bochum, Germany}
\email{sasander@me.com}

\begin{abstract}
Open sets are central to mathematics, especially analysis and topology, in ways few notions are.  
In most, if not all, computational approaches to mathematics, open sets are only studied indirectly via their `codes' or `representations'.
In this paper, we study how hard it is to compute, given an arbitrary open set of reals, the most common representation, i.e.\ a countable set of open intervals.  
We work in Kleene's higher-order computability theory, in particular its equivalent lambda calculus formulation due to Platek.
We establish many computational equivalences between on one hand the `structure' functional that converts open sets to the aforementioned representation, and on the other hand functionals 
arising from mainstream mathematics, like basic properties of semi-continuous functions, the Urysohn lemma, and the Tietze extension theorem. 
We also compare these functionals to known operations on regulated and bounded variation functions, and the Lebesgue measure restricted to closed sets.  We obtain a number of natural computational equivalences for the latter involving theorems from mainstream mathematics.  
\end{abstract}
%
\maketitle              

\section{Introduction}
\subsection{Motivation and overview}\label{intro}
First of all, \emph{computability theory} is a discipline in the intersection of theoretical computer science and mathematical logic where the fundamental question is as follows:
\begin{center}
\emph{given two mathematical objects $X $ and $ Y$, is $X$ computable from $Y$ {in principle}?}
\end{center} 
If $X $ and $Y$ are real numbers, Turing's `machine' model (\cite{tur37}) is the standard approach, i.e.\ `computation' is interpreted in the sense of Turing machines.  
To formalise computation involving (total) abstract objects, like functions on the real numbers or well-orderings of the reals, Kleene introduced his S1-S9 computation schemes in \cites{kleeneS1S9}.  
In Section~\ref{lambdaz}, we introduce a version\footnote{The authors have previously introduced a lambda calculus in \cite{dagsamXIII} similar to that in Section~\ref{lambdaz}.  Unfortunately, the calculus from \cite{dagsamXIII} contains a  technical error, which was recently discovered by John Longley.  This error does not affect the computability-theoretic results in \cite{dagsamXIII}, but does require fixing.  We discuss the exact nature of this fix in Section \ref{lambdaz} in some detail.}  of the lambda calculus involving fixed point operators that exactly captures S1-S9 and accommodates partial objects, 
with a brief introduction in Section \ref{prelim}.   
Henceforth, any reference to computability is to be understood in Kleene's framework and (if relevant) the extension to partial objects from Section \ref{lambdaz}.
For those not intimate with S1-S9, we hasten to add that many of our positive results, i.e.\ of the form \emph{$X$ is computable from $Y$}, are actually of the form \emph{$X$ is computable from $Y$ via a term of G\"odel's $T$ of low complexity}. However, these are only examples and we have no structural results at this point.  

\smallskip

Secondly, the notion of open (and closed) set is central to mathematics, especially topology and analysis, in ways few mathematical definitions are.  
Historically, the concept of open set dates back to Baire's doctoral thesis (\cite{beren2}), while Dedekind already studied such notions twenty years earlier; the associated paper \cite{didicol} was only published much later (\cite{moorethanudeserve}).
Now, for various reasons, the computational study of open sets usually takes place indirectly, namely via \emph{representations} or \emph{codes} for open sets, as discussed in detail in Remark \ref{XzX}.
In this light, it is a natural question how hard it is to compute such representations in terms of open sets.  
 In this paper, we study this question for the well-known representation of open sets of reals as countable unions of open intervals.  We identify a large number of operations that have the same computational complexity and compare the latter to known operations from \cite{dagsamXII, dagsamXIII} stemming from mainstream mathematics, including the Lebesgue measure restricted to closed sets.  We discuss some open problems related the previous in Section \ref{hopen}.  

\smallskip

Thirdly, in more detail, the following functional constitutes the central object of study in this paper.  
We note that open and closed sets are given by characteristic functions (see Section \ref{cdef}), well-known from e.g.\ probability theory.
\bdefi\label{bas}
The \emph{$\Omega_{C}$-functional} is defined exactly when the input $X\subset 2^\N$ is a closed set, in which case $\Omega_{C}(X)=0$ if $X=\emptyset$ and $\Omega_{C}(X)=1$ if $X\ne \emptyset$.
\edefi
We have shown in \cite{dagsamVII} that $\Omega_{C}$ is \emph{hard}\footnote{By e.g.\ \cite{dagsamVII}*{Theorem 6.6}, no type 2 functional can compute $\Omega_{C}$.  This includes the (type 2) functionals $\SS_{k}^{2}$ from Section \ref{lll} which decide $\Pi_{k}^{1}$-formulas.\label{grotevoet}} to compute.
In Section \ref{clust}, we shall introduce the \emph{$\Omega_{C}$-cluster}, a collection of functionals that are computationally equivalent (see Def.\ \ref{specs}) to $\Omega_{C}$, assuming Kleene's quantifier $\exists^{2}$ from Section \ref{lll}.  We identify a large number of elements of the $\Omega_{C}$-cluster, many stemming from mainstream mathematics, like properties of semi-continuous functions.  
The following list contains some representative members of the $\Omega_{C}$-cluster studied in Section \ref{clust}. 
\begin{itemize}
\item A partial functional $\Phi:(\R\di \R)\di \R$ such that for open $O\subset [0,1]$, $\Phi(O)$ is a code for $O$ (see \cite{simpson2}*{II.5.6} for the latter).
\item A partial functional $\Phi:(\R\di \R)\di \R$ such that for any upper semi-continuous $f:[0,1]\di \R$ and $p, q\in \Q\cap[0,1]$, $\Phi(f)=\sup_{x\in [p, q]}f(x)$. 
\item A partial functional $\Phi:(\R\di \R)\di \{0,1\}$ such that for any upper semi-continuous $f:[0,1]\di \R$, $\Phi(f)=0$ if and only $f$ is continuous on $[0,1]$.
\item Witnessing functionals for Urysohn lemma and Tietze extension theorem.
\end{itemize}
Interestingly, computing suprema or deciding continuity for \emph{arbitrary} functions requires Kleene's quantifier $\exists^{3}$.  
We obtain similar equivalences for the Lebesgue measure restricted to closed sets in Section \ref{LM}.  In this study, we shall always assume Kleene's $\exists^{2}$ from Section \ref{lll} to be given; 
we motivate this assumption by the observation that we mainly study \emph{discontinuous} functions on the reals, from which $\exists^{2}$ is readily obtained via \emph{Grilliot's trick} (\cite{grilling, kohlenbach2}). 

\smallskip

Fourth, we shall compare the properties of $\Omega_{C}$ to the $\Omega_{b}$-functional, where the latter was studied in \cite{dagsamXII, dagsamXIII}.
We note that we may replace $2^\N$ with $[0,1]$ or $\N^\N$ in Def.\ \ref{OK} without consequence.  At this moment, we do not know whether $\Omega_{C}$ and $\Omega_{b}$ are computationally equivalent, as reflected in Problem \ref{whatsyour} below.  
\bdefi\label{OK}
The $\Omega_{b}$-functional is defined exactly when the input $X\subset 2^\N$ has at most one element, in which case $\Omega_{b}(X)=0$ if $X=\emptyset$ and $\Omega_{b}(X)=1$ if $X\ne \emptyset$.
\edefi
Assuming Kleene's $\exists^{2}$ from Section \ref{lll}, the functional $\Omega_{b}$ is computationally equivalent to the functional $\Omega$ defined as follows (see \cite{dagsamXIII}*{Lemma 3.3.5})
\bdefi 
The functional $\Omega(X)$ is defined if $|X| \leq 1$ and outputs the element of $X$ when $|X| = 1$ and the constant $0^1$ when $X = \emptyset$.

\edefi
We have previously studied the $\Omega_{b}$-cluster (called `$\Omega$-cluster' in \cite{dagsamXIII}), which is the collection of functionals that are computationally equivalent (see Def.\ \ref{specs}) to $\Omega_{b}$, assuming $\exists^{2}$.
A number of natural operations on bounded variation and regulated functions are in the $\Omega_{b}$-cluster.  
We shall exhibit notable differences and similarities between the $\Omega_{b}$- and $\Omega_{C}$-cluster.
One important similarity is that $\Omega_{C}$ is \emph{lame} by Theorem \ref{tame}, i.e.\ the combination $\Omega_{C}+\exists^{2}$ computes the same real numbers as Kleene's $\exists^{2}$.  
By Theorem \ref{hench}, there are even no total functionals `in between' $\Omega_{C}$ and $\Omega_{b}$ from the computational point of view.  
Nonetheless, we identify a number of natural \emph{partial} functionals in Section \ref{betwi} that \emph{are} intermediate in this way.
An important difference between $\Omega_{b}$ and $\Omega_{C}$ is that they witness basic properties of two different function classes, namely the bounded variation and semi-continuous functions, although there is overlap between the latter. 

\smallskip

On a conceptual note, $\Omega_C$ and $\Omega_{b}$ are examples of what we call \emph{a structure functional}, i.e.\ a functional that does not turn up as a result of some construction in mainstream mathematics, but which is nonetheless useful for calibrating the computational complexity of those that do, like the functionals associated to Urysohn's lemma and the Tietze extension theorem. 

\smallskip

Figure \ref{xxx} below shows some connections between $\Omega_C$, $\Omega_{b}$, and related functionals, like the Lebesgue measure $\blambda_{C}$ restricted to closed sets.  
Moreover, Baire and (weak) Cantor realisers are witnessing functionals for the Baire category theorem and the uncountability of $\R$, while a Cantor intersection functional is a witnessing functional for the Cantor intersection theorem (see Def.\ \ref{keffer} and \ref{keffer2}).  
All functionals are computable in $\exists^{3}$ and strictly weaker than the latter. 
\begin{figure}[H]
\begin{tikzpicture}
  \matrix (m) [matrix of math nodes,row sep=3em,column sep=4em,minimum width=2em]
  {
   ~ & \Omega_{C}& \textup{Baire realiser}\\  
   \Omega_{b} &  \blambda_{C}~& \begin{array}{c}\textup{Cantor intersection}\\\textup{functional}\end{array} \\  
~&  \textup{Cantor realiser} &\textup{weak Cantor realiser} \\  
   };
  \path[-stealth]
    (m-1-2) edge node [left] {} (m-2-1)
        (m-1-2) edge node [left] {} (m-1-3)
        (m-1-3) edge [bend right=20] (m-3-2)
                (m-1-2) edge node [left] {} (m-2-2)
        (m-2-3) edge node [strike out, draw, -] {} (m-1-2)
                         (m-1-2) edge [bend right=10] (m-2-3)
        (m-2-2) edge node [strike out, draw, -] {} (m-2-1)
        (m-2-2) edge node [left] {?} (m-3-2)
                (m-2-1) edge node [left] {} (m-3-2)
                 (m-3-2) edge [bend right=10] (m-3-3)
                 (m-3-3) edge node [strike out, draw, -] {} (m-3-2)
               (m-2-3) edge node [left] {} (m-3-3)
;
\end{tikzpicture}
\caption{Some relations among our functionals}
\label{xxx}
\end{figure}
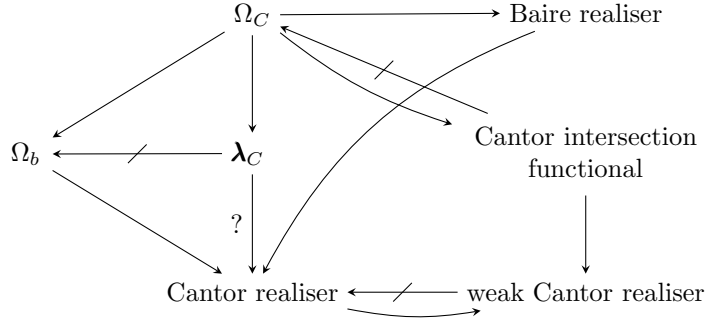
Finally, open sets are studied indirectly via `codes' or `representations' in computational approaches to mathematics. 
The following historical remark provides an overview of some such computational approaches to open sets.
\begin{rem}[Open sets and representations]\label{XzX}\rm
A set is \emph{open} if it contains a neighbourhood around each of its points, and an open set can be represented as a \emph{countable} union of such neighbourhoods in separable spaces, a result going back to Cantor.    
The latter and similar representations are used in various `computational' approaches to mathematics, as follows. 

\smallskip

For instance, the neighbourhood around a point of an open set is often assumed to be given together with this point (see e.g.\ \cite{bish1}*{p.\ 69}).  This is captured by our representation R2 studied in \cite{dagsamVII}*{\S7} and below.
Alternatively, open sets are simply represented as countable unions -called `codes' in \cite{simpson2}*{II.5.6}, `names' in \cite{wierook}*{\S1.3.4}, `witnesses' in \cite{weverketoch}, and `presentations' in \cite{littlefef}- of basic open neighbourhoods.
A related notion is \emph{locatedness} which means the (continuous) distance function $d(a, A):=\inf_{b\in A}d(a, b)$ exists for the set $A\subset \R$ (see \cite{bish1}*{p.\ 82}, \cite{withgusto}, or \cite{twiertrots}*{p.\ 258}), and numerous sufficient conditions are known (\cite{withgusto}).    

\smallskip

The representation of open sets as countable unions has the advantage that one can (effectively) switch between codes for open sets and codes for continuous functions (see e.g.\ \cite{simpson2}*{II.7.1}).  
As shown in \cite{dagsamVII, kohlenbach4, dagsamXIV}, computing a code for a continuous function, say on Cantor or Baire space or the unit interval, is \emph{rather easy}, while computing a code for an open set is \emph{rather hard}, in the sense of Footnote~\ref{grotevoet}.  
\end{rem}

\subsection{Preliminaries and definitions}\label{kelim}
We briefly introduce Kleene's \emph{higher-order computability theory} in Section~\ref{prelim}.
We introduce some essential axioms (Section~\ref{lll}) and definitions (Section~\ref{cdef}).  A full introduction may be found in e.g.\ \cite{dagsamX}*{\S2} or \cite{longmann}.

\smallskip

Since Kleene's computability theory borrows heavily from type theory, we shall often use common notations from the latter; for instance, the natural numbers are type $0$ objects, denoted $n^{0}$ or $n\in \N$.  
Similarly, elements of Baire space are type $1$ objects, denoted $f\in \N^{\N}$ or $f^{1}$.  Mappings from Baire space $\N^{\N}$ to $\N$ are denoted $Y:\N^{\N}\di \N$ or $Y^{2}$. 
An overview of such notations can be found in e.g.\ \cite{longmann, dagsamXIII}. 

\subsubsection{Kleene's computability theory}\label{prelim}
Our main results are in computability theory and we make our notion of `computability' precise as follows.  
\begin{enumerate}
\item[(I)] We adopt $\ZFC$, i.e.\ Zermelo-Fraenkel set theory with the Axiom of Choice, as the official metatheory for all results, unless explicitly stated otherwise.
\item[(II)] We adopt Kleene's notion of \emph{higher-order computation} as given by his nine clauses S1-S9 (see \cite{longmann}*{Ch.\ 5} or \cite{kleeneS1S9}) as our official notion of `computable' involving total objects. \footnote{We stress that Kleene's S1-S9 schemes make perfect sense for partial functionals of \emph{pure} type.  Since we may view the total functionals as a substructure of the partial ones, no extra computational strength, for total inputs, is added when extending the interpretation of Kleene's S1-S9 schemes to partial functionals.}
\end{enumerate}
We mention that S1-S8 are rather basic and merely introduce a kind of higher-order primitive recursion with higher-order parameters. 
The real power comes from S9, which essentially hard-codes the \emph{recursion theorem} for S1-S9-computability in an ad hoc way.  
By contrast, the recursion theorem for Turing machines is derived from first principles in \cite{zweer}.

\medskip

On a historical note, it is part of the folklore of computability theory that many have tried (and failed) to formulate models of computation for objects of all finite types and in which one derives the recursion theorem in a natural way.  For this reason, Kleene ultimately introduced S1-S9, which 
were initially criticised for their aforementioned ad hoc nature, but eventually received general acceptance.  
Now, the authors have introduced a new computational model based on the lambda calculus in \cite{dagsamXIII} with the following properties:
\begin{itemize}
\item S1-S8 is included while the `ad hoc' scheme S9 is replaced by more natural (least) fixed point operators,
\item the new model exactly captures S1-S9 computability for total objects,
\item the new model accommodates `computing with partial objects',
\item the new model is more modular than S1-S9 in that sub-models are readily obtained by leaving out certain fixed point operators.
\end{itemize}
As noted in Section \ref{intro}, the lambda calculus from \cite{dagsamXIII} contains a  technical error, which was recently discovered by John Longley.  
We discuss a fix in Section~\ref{lambdaz}. This involves essentially replacing the calculus from \cite{dagsamXIII} with one essentially from Platek \cite{Pla}   and recovering the computation tree characterisation from \cite{dagsamXIII} as Theorem \ref{thm5.7}.


\smallskip

Next, we mention the distinction between `normal' and `non-normal' functionals  based on the following definition from \cite{longmann}*{\S5.4}. 
We only make use of $\exists^{n}$ for $n=2,3$, as defined in Section \ref{lll}.
\bdefi\label{norma}
For $n\geq 2$, a functional of type $n$ is called \emph{normal} if it computes Kleene's quantifier $\exists^{n}$ following S1-S9, and \emph{non-normal} otherwise.  
\edefi
\noindent
It is a historical fact that higher-order computability theory, based on Kleene's S1-S9, has focused primarily on the world of \emph{normal} functionals; this observation can be found \cite{longmann}*{\S5.4} and can be explained by the (then) relative scarcity of interesting non-normal functionals, like the \emph{fan functional}, originally due to Kreisel (see \cite{dagtait} for historical details).  
The authors have recently identified interesting \emph{non-normal} functionals, namely those that compute the objects claimed to exist by:
\begin{itemize}
\item covering theorems due to Heine-Borel, Vitali, and Lindel\"of (\cites{dagsamV}),
\item the Baire category theorem (\cite{dagsamVII, samcsl23}),
\item local-global principles like \emph{Pincherle's theorem} (\cite{dagsamV}),
\item weak fragments of the Axiom of (countable) Choice (\cite{dagsamIX}),
\item the Jordan decomposition theorem and related results (\cites{dagsamXII, dagsamXIII}),
\item the uncountability of $\R$ (\cites{dagsamX, dagsamXI}).
\end{itemize}
This paper continues the study of non-normal functionals that originate from basic properties of closed sets and semi-continuous functions.  

\smallskip

Finally, the first example of a non-computable non-normal functional, Kreisel's (aka Tait's) fan functional (see \cite{dagtait}), is rather tame: it is computable in $\exists^{2}$. 
By contrast, the functionals based on the previous list, including the operation \eqref{ting} from Section~\ref{intro}, are computable in $\exists^{3}$ but not computable in any $\SS_{k}^{2}$, where the latter decides $\Pi_{k}^{1}$-formulas (see Section \ref{lll} for details).

\subsubsection{Some comprehension functionals}\label{lll}
In Turing-style computability theory, computational hardness is measured in terms of where the oracle set fits in the well-known comprehension hierarchy.  
For this reason, we introduce some functionals related to \emph{higher-order comprehension} in this section.
We are mostly dealing with \emph{conventional} comprehension here, i.e.\ only parameters over $\N$ and $\N^{\N}$ are allowed in formula classes like $\Pi_{k}^{1}$ and $\Sigma_{k}^{1}$.  

\medskip
\noindent
First of all, \emph{Kleene's quantifier $\exists^{2}: \N^{\N}\di \{0,1\} $} is the unique functional satisfying: 
\be\label{muk}
(\forall f^{1})\big[(\exists n)(f(n)=0) \asa \exists^{2}(f)=0    \big]. 
\ee
Clearly, $\exists^{2}$ is discontinuous at $f=11\dots$ in the usual epsilon-delta sense. 
In fact, given a discontinuous function on $\N^{\N}$ or $\R$, \emph{Grilliot's trick} computes $\exists^{2}$ from the former, via a rather low-level term of G\"odel's $T$ (see \cite{kohlenbach2}*{\S3}).
Moreover, $\exists^{2}$ computes \emph{Feferman's $\mu^{2}$} defined for any $f^{1}$ as follows:
\be\label{muk2}
\mu(f):=
\begin{cases}
n & \textup{if $n^{0}$ is the least natural number such that $f(n)=0$}\\
0 & \textup{if there are no $m^{0}$ such that $f(m)=0$}
\end{cases}.
\ee
Hilbert and Bernays formalise considerable swaths of mathematics using only $\mu^{2}$ (with that name) in \cite{hillebilly2}*{Supplement IV}.

\medskip
\noindent
Secondly, the \emph{Suslin functional} $\SS^{2}:\N^{\N}\di \{0,1\}$ (see \cites{kohlenbach2, avi2}) is the unique functional satisfying the following, for any $f^{1}$:
\be\label{muk3}
  (\exists g^{1})(\forall n^{0})(f(\overline{g}n)=0)\asa \SS(f)=0.
\ee
By definition, the Suslin functional $\SS^{2}$ can decide whether a $\Sigma_{1}^{1}$-formula as in the left-hand side of \eqref{muk3} is true or false.   
We similarly define the functional $\SS_{k}^{2}$ which decides the truth or falsity of $\Sigma_{k}^{1}$-formulas, given in their Kleene normal form (see e.g.\ \cite{simpson2}*{IV.1.4}).
The Feferman-Sieg operators $\nu_{n}$ from \cite{boekskeopendoen}*{p.\ 129} are essentially $\SS_{n}^{2}$ strengthened to return a witness (if existent) to the $\Sigma_{n}^{1}$-formula at hand.  

\medskip

\noindent
Thirdly,  \emph{Kleene's quantifier $\exists^{3}: (\N^{\N}\di \N)\di \{0,1\} $} is the functional satisfying: 
\be\label{muk4}
(\forall Y^{2})\big[  (\exists f^{1})(Y(f)=0)\asa \exists^{3}(Y)=0  \big].
\ee
Hilbert and Bernays introduce functionals in e.g.\ \cite{hillebilly2}*{Supplement IV, p.\ 479} that readily compute $\exists^{3}$.  

\medskip

In conclusion, the functionals in Figure \ref{xxx} from Section \ref{intro} are computable in $\exists^{3}$ but not in any type 2 functional, which includes $\SS_{k}^{2}$ from Section \ref{lll} (see \cites{dagsamXII, dagsamXI, dagsamVII} for proofs).
Many non-normal functionals exhibit the same `computational hardness' and we merely view this as support for the development of a separate scale for classifying non-normal functionals.    

\subsubsection{Some definitions}\label{cdef}
We introduce some definitions needed in the below, mostly stemming from mainstream mathematics.
We note that subsets of $\R$ are given by their characteristic functions (Definition \ref{char}), where the latter are common in measure and probability theory.

\medskip
\noindent
First of all, we make use the usual definition of (open) set, where $B(x, r)$ is the open ball with radius $r>0$ centred at $x\in \R$.
Note that `RM' stands for `Reverse Mathematics' in the final items. 
\bdefi[Set]\label{char}~
\begin{itemize}
\item Subsets $A$ of $ \R$ are given by their characteristic function $F_{A}:\R\di \{0,1\}$, i.e.\ we write $x\in A$ for $ F_{A}(x)=1$ for all $x\in \R$.
\item We write `$A\subset B$' if we have $F_{A}(x)\leq F_{B}(x)$ for all $x\in \R$. 
\item A subset $O\subset \R$ is \emph{open} in case $x\in O$ implies that there is $k\in \N$ such that $B(x, \frac{1}{2^{k}})\subset O$.
\item A subset $C\subset \R$ is \emph{closed} if the complement $\R\setminus C$ is open. 
\item A subset $O\subset \R$ is \emph{R2-open} in case there is $Y:\R\di \R$ such that 
$x\in O$ implies $Y(x)>0\wedge B(x, Y(x))\subset O$ and $Y(x)=0\di x\not\in O$, for any $x\in \R$.  
\item A subset $C\subset \R$ is \emph{R2-closed} if the complement $\R\setminus C$ is R2-open. 
\item A subset $O\subset \R$ is \emph{RM-open} if there are sequences $(a_{n})_{n\in \N}, (b_{n})_{n\in \N}$ of reals such that $O=\cup_{n\in \N}(a_{n}, b_{n})$.
\item A subset $C\subset \R$ is \emph{RM-closed} if the complement $\R\setminus C$ is RM-open.  
\end{itemize}
\edefi
\noindent
No computational data/additional representation is assumed in our definition of open set.  
As established in \cites{dagsamXII, dagsamXIII, samBIG}, one readily comes across closed sets in basic real analysis (Fourier series) that come with no additional representation. 
The function $Y:\R\di \R$ is called an \emph{R2-representation} for the R2-open set $O$ while the sequences $(a_{n})_{n\in \N}, (b_{n})_{n\in \N}$ are called an \emph{R4-representation} (see \cite{dagsamVII}) or \emph{RM-code} for the RM-open set $O$. 

\smallskip

Secondly, the following sets are often crucial in proofs in real analysis. 
\bdefi
The sets $C_{f}$ and $D_{f}$ respectively gather the points where $f:\R\di \R$ is continuous and discontinuous.
\edefi
One problem with $C_{f}, D_{f}$ is that `$x\in C_{f}$' involves quantifiers over $\R$.  
In general, deciding whether a given $\R\di \R$-function is continuous at a given real, is as hard as $\exists^{3}$ from Section \ref{lll}.
For these reasons, the sets $C_{f}, D_{f}$ do exist in general, but are not computable in e.g.\ $\exists^{2}$.  For quasi-continuous and semi-continuous functions, these sets are definable in $\exists^{2}$ by \cite{dagsamXIII}*{\S2} or \cite{samBIG2}*{Theorem 2.4}, a fact we use often.

\smallskip

Thirdly, to avoid the problem sketched in the previous formula, one can additionally assume the existence of the oscillation function $\osc_{f}:\R\di \R$ as in Def.\ \ref{oscf}.  The formula `$x\in C_{f}$' is then equivalent to the \emph{arithmetical} formula $\osc_{f}(x)=_{\R}0$. 
\bdefi[Oscillation function]\label{oscf}
For any $f:\R\di \R$, the associated \emph{oscillation functions} are defined as follows: $\osc_{f}([a,b]):= \sup _{{x\in [a,b]}}f(x)-\inf _{{x\in [a,b]}}f(x)$ and $\osc_{f}(x):=\lim _{k \di \infty }\osc_{f}(B(x, \frac{1}{2^{k}}) ).$
\edefi
We note that Riemann, Hankel, and Ascoli already study the notion of oscillation in the context of Riemann integration (\cites{hankelwoot, rieal, ascoli1}).  

\smallskip

Fourth, we shall study the following notions, many of which are well-known and hark back to the days of Baire, Darboux, Hankel, and Volterra (\cites{beren,beren2,darb, volaarde2,hankelwoot,hankelijkheid}).  
\bdefi\label{flung} 
For $f:[0,1]\di \R$, we have the following definitions:
\begin{itemize}
\item $f$ is \emph{upper semi-continuous} at $x_{0}\in [0,1]$ if $f(x_{0})\geq_{\R}\lim\sup_{x\di x_{0}} f(x)$,
\item $f$ is \emph{lower semi-continuous} at $x_{0}\in [0,1]$ if $f(x_{0})\leq_{\R}\lim\inf_{x\di x_{0}} f(x)$,
\item $f$ is \emph{quasi-continuous} at $x_{0}\in [0, 1]$ if for $ \epsilon > 0$ and an open neighbourhood $U$ of $x_{0}$, 
there is a non-empty open ${ G\subset U}$ with $(\forall x\in G) (|f(x_{0})-f(x)|<\eps)$.
\item $f$ is \emph{cliquish} at $x_{0}\in [0, 1]$ if for $ \epsilon > 0$ and an open neighbourhood $U$ of $x_{0}$, 
there is a non-empty open ${ G\subset U}$ with $(\forall y, z\in G) (|f(y)-f(z)|<\eps)$.
\item $f$ is \emph{regulated} if for every $x_{0}$ in the domain, the `left' and `right' limit $f(x_{0}-)=\lim_{x\di x_{0}-}f(x)$ and $f(x_{0}+)=\lim_{x\di x_{0}+}f(x)$ exist.  
\item $f$ is \emph{Baire 0} if it is a continuous function. 
\item $f$ is \emph{Baire $n+1$} if it is the pointwise limit of a sequence of Baire $n$ functions.
\item $f$ is \emph{Baire 1$^{*}$} if\footnote{The notion of Baire 1$^{*}$ goes back to \cite{ellis} and equivalent definitions may be found in \cite{kerkje}.  
In particular,  Baire 1$^{*}$ is equivalent to the Jayne-Rogers notion of \emph{piecewise continuity} from \cite{JR}.} there is a sequence of closed sets $(C_{n})_{n\in \N}$ such that $[0,1]=\cup_{n\in \N}C_{n}$ and $f_{\upharpoonright C_{m}}$ is continuous for all $m\in \N$.
\item $f$ is \emph{continuous almost everywhere} if it is continuous at all $x\in [0,1]\setminus E$, where $E$ is a measure zero\footnote{A set $A\subset \R$ is \emph{measure zero} if for any $\eps>0$ there is a sequence of basic open intervals $(I_{n})_{n\in \N}$ such that $\cup_{n\in \N}I_{n}$ covers $A$ and has total length below $\eps$.  Note that this notion does not depend on (the existence of) the Lebesgue measure.} set.
\item $f$ is \emph{pointwise discontinuous} if for any $x\in [0,1]$ and $\eps>0$, there is $y\in [0,1]$ such that $f$ is continuous at $y$ and $|x-y|<\eps$ (Hankel, 1870, \cite{hankelwoot}).  
\end{itemize}
\edefi
As to notations, a common abbreviation is `usco' and `lsco' for the first two items.  
Moreover, if a function has a certain weak continuity property at all reals in $[0,1]$ (or its intended domain), we say that the function has that property.  
We shall generally study $\R\di \R$-functions but note that `usco' and `lsco' are also well-defined for $f:[0,1]\di \overline{\R}$ where $\overline{\R}=\R\cup\{+\infty, -\infty\}$ involves 
two special symbols that satisfy $(\forall x\in \R)(-\infty <_{\R} x <_{\R}+\infty )$ by fiat.     

\smallskip

We note that cliquish functions are exactly those functions that can be expressed as the sum of two quasi-continuous functions (\cite{quasibor2}).  
Nonetheless, these notions can behave fundamentally different (see e.g.\ \cite{dagsamXIII}*{\S2.8}).  

\smallskip

Fifth, the notion of \emph{bounded variation} (abbreviated $BV$) was first explicitly\footnote{Lakatos in \cite{laktose}*{p.\ 148} claims that Jordan did not invent or introduce the notion of bounded variation in \cite{jordel}, but rather discovered it in Dirichlet's 1829 paper \cite{didi3}.} introduced by Jordan around 1881 (\cite{jordel}) yielding a generalisation of Dirichlet's convergence theorems for Fourier series.  
Indeed, Dirichlet's convergence results are restricted to functions that are continuous except at a finite number of points, while functions of bounded variation can have (at most) countable many points of discontinuity, as already studied by Jordan, namely in \cite{jordel}*{p.\ 230}.
Nowadays, the \emph{total variation} of $f:[a, b]\di \R$ is defined as follows:
\be\label{tomb}\textstyle
V_{a}^{b}(f):=\sup_{a\leq x_{0}< \dots< x_{n}\leq b}\sum_{i=0}^{n-1} |f(x_{i})-f(x_{i+1})|.
\ee
If this quantity exists and is finite, one says that $f$ has bounded variation on $[a,b]$.
Now, the notion of bounded variation is defined in \cite{nieyo} \emph{without} mentioning the supremum in \eqref{tomb}; see also \cites{kreupel, briva, brima}.  
Hence, we shall distinguish between the following notions.  Jordan seems to use item \eqref{donp} of Definition~\ref{varvar} in \cite{jordel}*{p.\ 228-229}.
\bdefi[Variations on variation]\label{varvar}
\begin{enumerate}  
\renewcommand{\theenumi}{\alph{enumi}}
\item The function $f:[a,b]\di \R$ \emph{has bounded variation} on $[a,b]$ if there is $k_{0}\in \N$ such that $k_{0}\geq \sum_{i=0}^{n-1} |f(x_{i})-f(x_{i+1})|$ 
for any partition $x_{0}=a <x_{1}< \dots< x_{n-1}<x_{n}=b  $.\label{donp}
\item The function $f:[a,b]\di \R$ \emph{has {a} variation} on $[a,b]$ if the supremum in \eqref{tomb} exists and is finite.\label{donp2}
\end{enumerate}
\edefi
\noindent
The fundamental theorem about $BV$-functions (see e.g.\  \cite{jordel}*{p.\ 229}) is as follows.
\begin{thm}[Jordan decomposition theorem]\label{drd}
A function $f : [0, 1] \di \R$ of bounded variation is the difference of two non-decreasing functions $g, h:[0,1]\di \R$.
\end{thm}
Theorem \ref{drd} has been studied extensively via second-order representations in e.g.\ \cites{groeneberg, kreupel, nieyo, verzengend}.
The same holds for constructive analysis by \cites{briva, varijo,brima, baathetniet}, involving different (but related) constructive enrichments.  
Now, arithmetical comprehension suffices to derive Theorem \ref{drd} for various kinds of second-order \emph{representations} of $BV$-functions in \cite{kreupel, nieyo}.
By contrast, the results in \cite{dagsamXII,dagsamX, dagsamXI, dagsamXIII} show that the Jordan decomposition theorem is even `explosive': combining with the Suslin $\SS^{2}$ (see Section \ref{lll}), one derives the much stronger $\SS_{2}^{2}$.  
 
\smallskip

Finally, we have previously studied the following functionals in e.g.\ \cites{dagsamVII,dagsamXIII, dagsamXII}, also found in Figure \ref{xxx} in Section \ref{intro}.  
\bdefi[Some functionals]\label{keffer}~
\begin{itemize}
\item A \emph{Baire realiser} is defined for input $(O_{n})_{n\in \N}$ being a sequence of open and dense sets in $[0,1]$; the output is any $x\in \cap_{n\in \N}O_{n}$. 
\item A \emph{Cantor realiser} is defined for input $(A, Y)$ where $Y:\R\di \N$ is injective on $A\subset [0,1]$; the output is any real $x\in [0,1]\setminus A$. 
\item A \emph{\textbf{weak} Cantor realiser} is defined for input $(A, Y)$ where $Y:\R\di \N$ is injective on $A\subset [0,1]$ and satisfies $(\forall n\in \N)(\exists x\in A)(Y(x)=n)$; the output is any real $x\in [0,1]\setminus A$. 
\end{itemize}
\edefi

\section{Populating the $\Omega_{C}$-cluster}\label{clust}
\subsection{Introduction}
In this section, we introduce the $\Omega_{C}$-cluster and establish that many 
functionals originating from mainstream mathematics belong to this cluster.  For instance, many functionals witnessing basic properties of semi-continuous functions turn out to be members of the $\Omega_{C}$-cluster, as well as functionals witnessing the Urysohn lemma and Tietze extension theorem.

\smallskip

First of all, we recall the definition of $\Omega_{C}$ from Section \ref{intro}.
\bdefi\label{bass}
The \emph{$\Omega_{C}$-functional} is defined exactly when the input $X\subset 2^\N$ is a closed set, in which case $\Omega_{C}(X)=0$ if $X=\emptyset$ and $\Omega_{C}(X)=1$ if $X\ne \emptyset$.
\edefi
\noindent
The $\Omega_{C}$-cluster is now defined as the following equivalence class. 
\bdefi\label{specs}
We say that \emph{the functional $\Phi^{3}$ belongs to the $\Omega_{C}$-cluster} in case 
\begin{itemize}
\item the combination $\Phi+\exists^{2}$ computes $\Omega_{C}$, and 
\item the combination $\Omega_{C}+\exists^{2}$ computes $\Phi$. 
\end{itemize}
We also say that \emph{the functionals $\Phi$ and $\Omega_{C}$ are computationally equivalent given $\exists^{2}$}.
\edefi
Secondly, to avoid complicated definitions and domain restrictions, we will sometimes abuse notation and make statements of the form
\begin{center}
\emph{any functional $\Psi$ satisfying a given specification $\textsf{\textup{(A)}}$ belongs to the $\Omega_{C}$-cluster.}
\end{center}
The centred statement means that for \textbf{any} functional $\Psi_{0}$ satisfying the given specification $\textsf{(A)}$, the combination $\Psi_{0}+\exists^{2}$ computes $\Omega_{C}$, while $\Omega_{C}+\exists^{2}$ computes \textbf{some} functional $\Psi_{1}$ satisfying the specification $\Psi_{1}$.  

\smallskip
\noindent
Finally, as to content, this section is structured as follows.
\begin{itemize}
\item In Section \ref{clute}, we identify a large number of functionals in the $\Omega_{C}$-cluster, grouped in the \emph{first and second cluster theorem}. 
These are all generally related to semi-continuity and properties of closed/compact sets.
\item In Section \ref{relo}, we identify some functionals in or related to the $\Omega_{C}$-cluster that deal with the following (in our opinion) `unexpected' topics:
\begin{itemize}
\item basic properties of \emph{arbitrary functions} on the reals (Section \ref{trini}),
\item a generalisation of moduli of continuity  (Section \ref{mocosec}), 
\item computational equivalences assuming $\SS^{2}$ (Section \ref{trinix}).
\end{itemize}
\end{itemize}
Regarding the second bullet point, it is known that many basic properties of arbitrary functions on the reals are only decidable given $\exists^{3}$.  
The crux is of course that we additionally assume an {oscillation function} to be given in Section \ref{trini}.  
The latter function is always upper semi-continuous, explaining the connection to $\Omega_{C}$.

\subsection{The cluster theorems}\label{clute}
We establish the first (Theorem \ref{ficlu}) and second (Theorem \ref{seclu}) cluster theorems, identifying a large number of members of the $\Omega_{C}$-cluster.  
Many of these members are based on mainstream results pertaining to semi-continuous functions and exhibit some robustness. 
We assume that the functionals $\Phi_{1}$-$\Phi_{7}$ in Theorem \ref{ficlu} are undefined outside of their (clearly specified) domain of definition to avoid complicated specifications.  
\begin{thm}[First cluster theorem]\label{ficlu}
The following belong to the $\Omega_{C}$-cluster.  
\begin{itemize}
\item The partial functional $\Phi_0$ that to a \textbf{compact} subset $X$ of $\N^\N$ outputs 0 if $X$ is empty and 1 otherwise.
\item The partial functional $\Phi_1$ that to a closed subset $X$ of $[0,1]$ outputs 0 if $X$ is empty and 1 otherwise.
\item The partial functional $\Phi_2$  that to an open set $O \subset [0,1]$ outputs the set of pairs $\langle p,q\rangle$ of rational numbers such that $(p,q) \subseteq O$.
\item The partial functional $\Phi_{2,c}$ that to an open set $O \subset [0,1]$ outputs the set of pairs $\langle p,q\rangle$ of rational numbers such that $[p,q] \subseteq O$.
\item The partial functional $\Phi_3$ that to usco $f:[0,1]\di \R$ outputs $\sup_{x\in [0,1]} f(x)$.\label{krem}
\item  Any partial functional $\Phi_4$ that to usco $f:[0,1]\di \R$ outputs $y\in [0,1]$ where $f$ attains its supremum.
\item Any selector $\Phi_5$ for compact $X\subset 2^\N$, i.e.\ $\Phi_5(X) \in 2^\N$ for all compact $X$, and $\Phi_5(X) \in X$ when $X \neq \emptyset$.
\item Any functional $\Phi_{5, c}$ that to closed $C\subset [0,1]$ and $k\in \N$, outputs distinct $x_{0}, \dots, x_{k}\in C$ if such there are, and $0$ otherwise.  
\item Any functional $\Phi_6$ that to a perfect subset of $[0,1]$ outputs its optimal RM-code, alternatively \(and equivalently\) decides if a set that is either perfect or empty, is empty or perfect.
\item The functional $\Phi_7$ that to a closed subset $X$ of $2^\N$  outputs its optimal code.
\end{itemize}
\end{thm}
\begin{proof}
Zeroth, a closed subset of $2^\N$ is a compact subset of $\N^\N$, so $\Omega_{C}$ is a sub-function of $\Phi_0$ and thus computable in it. For the other direction we use that there is a Turing-computable injection from $\N^\N$ to $2^\N$ with a partial Turing-computable inverse. Using this, we readily compute $\Phi_0$ from $\Omega_{C}$

\smallskip
 
First of all, if $X \subseteq [0,1]$ is closed, then the set of binary expansions of the elements in $X$ is a compact subset of $2^\N$, so we can use $\Omega_C$ to decide if $X$ is non-empty or not. In this way, $\Omega_{C}$ computes $\Phi_{1}$ and the other direction is trivial.  Without loss of generality, we shall often write $\Omega_{C}$ for $\Phi_1$.

\smallskip

Secondly, if $O$ is an open subset of $[0,1]$ and $r < q$ are rational points in $[0,1]$, we can use $\Phi_1$ to decide if $[r,q] \subseteq O$ or not. Using $\exists^2$, this information yields $\Phi_2(O)$. So $\Phi_{1}$ computes $\Phi_{2}$, and by the same argument $\Phi_{2, c}$. Conversely, if $X \subseteq [0,1]$ is closed and $\Phi_2$ provides us with an RM-code for $[0,1] \setminus X$, then $\exists^2$ can check if this code contains a finite sub-covering of $[0,1]$ or not, 
i.e.\ if $X$ is empty or not. Thus, $\Phi_{2}$ computes $\Phi_{1}$, and so does $\Phi_{2, c}$. 

\smallskip

Thirdly, if $f:[0,1]\di \R$ is  usco, then for any $r\in \Q$ we have that $\{x \in [0,1] : f(x) \geq r\}$ is closed. Using this and $\exists^2$, $\Phi_1$ computes the Dedekind cut of $\sup f$, i.e.\ $\Phi_{1}$ computes $\Phi_{3}$.
If $X \subseteq 2^\N$ is compact, we consider $X$ as a closed subset of the (traditional) Cantor set. Then the characteristic function of $X$ is usco. Using $\Phi_3$ to find the sup of this function, we can decide if $X$ is empty or not, thus computing $\Omega_C(X)$, i.e.\ $\Phi_{3}$ and $\Phi_{4}$ compute $\Omega_{C}$.  
For the reversal, let $f:[0,1] \rightarrow \R$ be usco and use $\Phi_3$ to find $a = \sup f$, and define $X := \{x : f(x) = a\}$. Then $X$ is closed and non-empty, i.e.\ $\Phi_2$ yields an RM-code for $X$. From this RM-code we can define the least element of $X$ arithmetically, i.e.\ using $\exists^2$, yielding $\Phi_{4}$.

\smallskip

Fourth, any $\Phi_5$ clearly computes $\Omega_C$, while the other direction proceeds as follows: let $X\subset 2^{\N}$ be compact and use $\Omega_{C}$ to compute the set of basic neighbourhoods in $2^\N$ disjoint from $X$.  Use $\Omega_{C}$ to decide if $X$ is non-empty and, if it is, locate the lexicographically least element of $X$ using the usual interval-halving technique.  The proof for $\Phi_{5, c}$ proceeds in the same way.  
\smallskip

Fifth, $\Phi_{6}$ is trivially computable from $\Omega_{C}$; for the other direction it suffices to show that if we can decide between perfect and empty sets, we can also decide if a closed set is empty or not. 
Thus, let $X \subseteq [0,1]$ be closed. If $X \cap \Q\neq \emptyset$, then $X \neq \emptyset$. If $X$ and $\Q$ are disjoint, then $X$ is homeomorphic to the set of binary representations for the elements of $X$, i.e.\ $X$ can be regarded as a closed subset of $2^\N$. Using the bijection between $2^\N$ and $2^\N\times 2^{\N}$, we consider the set $X \times 2^\N$. This set is perfect if $X$ is non-empty, and empty otherwise. We can then use $\Phi_6$ on this set to decide if $X$ is empty or not.  The final item is done in the same way.
\end{proof}
Regarding $\Phi_{3}$ in Theorem \ref{ficlu}, it goes without saying that finding suprema for usco functions amounts to the same as finding infima for lsco functions.
Moreover, the definition of `honest RM-code for an lsco function' from \cite{ekelhaft}*{\S5} amounts to assuming that $(\forall x\in B(a, r))(f(x)\geq q)$ is $\Sigma_{1}^{0}$ with parameters $r, q, a\in \Q$. 
The latter formula is clearly equivalent to $q\leq \inf_{x\in B(a, r)}f(x)$, i.e.\ the notion of honest code is connected to $\Phi_{3}$ from Theorem \ref{ficlu}.

\smallskip

Next, we establish the second cluster theorem as in Theorem \ref{seclu}, for which we first introduce the following definition, an essential part of \emph{Ekeland's variational principle} (\cite{oozeivar}), well-known from analysis. 
\bdefi[Critical point]
For $\eps>0$, $x\in [0,1]$ is an \emph{$\eps$-critical point} of $f$ if
\be\label{denkf}
(\forall y\in [0,1])\big[ ( \eps |x-y| \leq f(x)-f(y) )\di y=x\big].
\ee
\edefi
\noindent 
Clearly, $\Phi_{18}$ and $\Phi_{19}$ in Theorem \ref{seclu} express Ekeland's variational principle. 

\smallskip

Regarding $\Phi_{22}$, the latter witnesses \cite{myerson}*{Theorem 2} which states that a set $S\subset [0,1]$ is ${\bf F}_{\sigma}$ if and only if there is a Baire 1 $f:[0,1]\di \R$ such that $S=\{x\in [0,1]:f(x)\ne 0\}$. 
Similarly, $\Phi_{23}$ is based on the fact that the Baire 1 sets are exactly those in ${\bf F}_{\sigma}\cap {\bf G}_{\delta}$ (see \cite{SVR}*{Theorem 11.6}).  Moreover, $\Phi_{29}$ computes a Lebesgue number\footnote{A real $\delta>0$ is a \emph{Lebesgue number} for the open covering $(O_{n})_{n\in \N}$ of $[0,1]$ if for any subset $X\subset [0,1]$ with diameter $<\delta$, there is $m\in \N$ such that $X\subset O_{m}$.\label{iop}} for an enumerated covering of open sets.

\smallskip

We (again) assume that the functionals in Theorem \ref{seclu} are undefined outside of their (clearly specified) domain of definition.  
\begin{thm}[Second cluster theorem]\label{seclu}
The following belong to the $\Omega_{C}$-cluster. 
\begin{itemize}
\item The functional $\Phi_{8}$ that to a closed set $C\subset [0,1]$, outputs $\sup C$.\label{new1}
\item The functional $\Phi_{9}$ that to closed $C\subset [0,1]$ and continuous $f:C\di \R$, outputs $\sup_{x\in C} f(x)$.\label{new2}
\item The functional $\Phi_{10}$ that to closed $C\subset [0,1]$ and usco $f:[0,1]\di \R$, outputs $\sup_{x\in C} f(x)$.\label{new3}
\item \(Urysohn\) Any functional $\Phi_{11}$ that to closed disjoint $C_{0}, C_{1}\subset [0,1]$, outputs continuous $f:[0,1]\di \R$ with $x\in C_{i}\asa f(x)=i$ for $x\in [0,1]$, $i\leq 1$.\label{ton}
\item \(weak Urysohn\) Any `weak' functional $\Phi_{11, w}$ which is the previous item with `continuity' replaced by `Baire 1' or `quasi-continuity'.\label{tonw}
\item \(Tietze\) Any functional $\Phi_{12}$ that to closed $C\subset [0,1]$ and continuous $f:C\di \R$, outputs continuous $g:[0,1]\di \R$ such that $g(x)=f(x)$ for $x\in C$.\label{tonn}
\item \(Tietze-Haussdorf, \cites{tietze, hauzen}\) Any functional $\Phi_{13}$ that to closed $C\subset [0,1]$ and lsco $f:[0,1]\di \R$ continuous on $C$, outputs an increasing sequence $(f_{n})_{n\in \N}$ of continuous functions with pointwise limit $f$ and $f=f_{n}$ on $C$. \label{tietfull}
\item The functional $\Phi_{14}$ that for usco $f:[0,1]\di \R$, decides whether $C_{f}=\emptyset$. \label{dagwo}
\item The functional $\Phi_{15}$ that for usco $f:[0,1]\di \R$, decides whether $f\in BV$ and outputs $V_{0}^{1}(f)$ if so. \label{dagwo2}
\item Any functional $\Phi_{16}$ that to usco $f:[0,1]\di \R$, outputs a descreasing sequence of continuous functions $(f_{n})_{n\in \N}$ with pointwise limit $f$.\label{ton2}
\item Any functional $\Phi_{17}$ that to usco $f:[0,1]\di \R$, outputs a sequence of continuous functions $(f_{n})_{n\in \N}$ with infimum $f$ \(Dilworth lemma, \cite{dill}*{Lem.\ 4.1}\).\label{ton21}
\item \(Ekeland\) Any functional $\Phi_{18}$ that to lsco $f:[0,1]\di \R$ and $\eps>0$, outputs an $\eps$-critical point $x\in [0,1]$. \label{ton3}
\item \(Ekeland\) Any functional $\Phi_{19}$ that to lsco $f:[0,1]\di \R$, outputs $x\in [0,1]$ such that for all $\eps>0$, $x$ is $\eps$-critical.\label{ton4}
\item Any functional $\Phi_{20}$ that to usco $f:[0,1]\di \R$, outputs an \(honest\) RM-code $\Phi$ that equals $f$ on $[0,1]$.\label{ton5}
\item \(Hahn-Kat\v{e}tov-Tong insertion theorem \cite{hahn1, kate,tong}\) The functional $\Phi_{21}$ that to usco $f:[0,1]\di \R$ and lsco $g: [0,1]\di \R$ with $(\forall x\in [0,1])(f(x)\leq g(x))$, outputs continuous $h:[0,1]\di \R$ with $(\forall x\in [0,1])(f(x)\leq h(x)\leq g(x))$.\label{ton6}
\item Any functional $\Phi_{22}$ that to a sequence $(C_{n})_{n\in \N}$ of closed sets in $[0,1]$, outputs Baire 1 $f:[0,1]\di \R$ and its representation such that $\cup_{n\in \N}C_{n}=\{x\in [0,1]:f(x)\ne 0\}$.\label{myer}
\item Any functional $\Phi_{23}$ that to closed $C\subset [0,1]$, outputs increasing $f:\R\di \R$ such that $f'$ is continuous and $C=\{ x\in [0,1]: f'(x)=0 \}$ \(\cite{SVR}*{Ex.\ 1.P}\).
\item The functional $\Phi_{24}$ that to usco $f:[0,1]\di \R$, outputs the associated `rising sun function' $f_{\odot}$, i.e.\ the minimal decreasing $g:[0,1]\di \R$ such that $g\geq f$.\label{tim1}
\item The functional $\Phi_{25}$ that for usco $f:[0,1]\di \overline{\R}$, decides whether $f$ is bounded above by some $N\in \N$.
\item The functional $\Phi_{26}$ that for usco $f:[0,1]\di \overline{\R}$, decides whether $f$ is Riemann integrable on $[0,1]$.
\item The functional $\Phi_{27}$ that for closed $C\subset [0,1]$, decides whether $C$ is a closed interval or not.
\item \(Finkelstein \cite{flinkenstein}\) Any functional $\Phi_{28}$ that for lsco $f:[0,1]\di [0,+\infty)$, outputs continuous $h:([0,1]\times \R)\di [0, +\infty)$ such that for each $x\in [0,1]$, $\lambda y.h(x, y)$ is integrable and $f(x)=\int_{-\infty}^{+\infty}h(x, y)dy$.
\item  Any functional $\Phi_{29}$ that to an open covering $(O_{n})_{n\in \N}$ of $[0,1]$, outputs a \emph{Lebesgue number}$^{\ref{iop}}$ for the covering. 
\end{itemize}
\end{thm}
\begin{proof}
First of all, $\Phi_{8}$ can decide whether a closed $C\subset [0,1]$ is non-empty by checking $[0<\Phi_{8}(C\cup\{0\})] \vee [0\in C]$ using $\exists^{2}$. 
For the reversal, $\Phi_{2}$ from the first cluster theorem provides an RM-code and $\exists^{2}$ readily yields the supremum by \cite{simpson2}*{IV.2.11}.
Clearly, items $\Phi_{9}$ and $\Phi_{10}$ readily compute $\Phi_{8}$ by considering $f(x):=x$ while the other direction follows using $\exists^{2}$ by \cite{simpson2}*{IV.2.11} and $\Phi_{3}$ from the first cluster theorem.
For the latter, obtain a code for the closed set and use $\Phi_{3}$ on the intervals of the code.  

\smallskip

Secondly, the equivalence involving $\Phi_{11}$ is proved in \cite{dagsamVII}*{Theorem 5.3}.  For $\Phi_{11, w}$, it suffices to show that the latter computes $\Phi_{1}$. 
Hence, let $X\subset(0,1)$ be closed, put $C_{1}=X$ and $C_{0}=\{0\}$, and let $f_{X}:[0,1]\di \R$ be the Baire 1 (or quasi-continuous) function provided by $\Phi_{11, w}$ .  
By (the proof of) \cite{dagsamXIV}*{Theorem 2.9}, $\exists^{2}$ can compute $\sup_{x\in [0,1]}f_{X}(x)$, both for Baire 1 and quasi-continuous functions, and we have $X=\emptyset \asa [1>\sup_{x\in [0,1]}f_{X}(x)]$, where the right-hand side is decidable using $\exists^{2}$.  The Tietze functional $\Phi_{12}$ computes the Urysohn functional $\Phi_{11}$ as follows: for closed and disjoint $C_{0}, C_{1}\subset [0,1]$, define $C:= C_{0}\cup C_{1}$ and 
let $f:C\di \R$ be $0$ on $C_{0}$ and $1$ on $C_{1}$.  Now apply the former functional to obtain the latter.
To compute $\Phi_{12}$, use $\Omega_{C}$ to obtain an RM-code for $C\subset [0,1]$.  To obtain a code for $f:C\di \R$, the proof of \cite{dagsamXIV}*{Theorem 2.3}, relativises to the RM-code of $C$.  
With this code in place, the standard second-order proof of the Tietze extension theorem yields an RM-code defined on $[0,1]$ that equals $f$ on $C$.  
Clearly, $\mu^{2}$ converts this code into a third-order function.  

\smallskip

Thirdly, for $\Phi_{14}$, note that for usco $f:[0,1]\di \R$, we have for all $x\in [0,1]$ that 
\[\textstyle
x\in D_{f}\asa [\liminf_{z\di x}f(z)<f(x)],
\] 
\emph{and} we can compute this $\liminf$ using $\exists^{2}$.  
Now define the following set 
\be\label{tago}\textstyle
D_{f, q}:=\{x\in [0,1]: \liminf_{z\di x}f(z)\leq f(x)-q\}.
\ee
Then $D_{f, q}$ is closed (since $f$ is usco) and $\cup_{q\in \Q^{+}}D_{f, q}=D_{f}$.  
Now use $\Omega_{C}+\exists^{2}$ to check whether $(\exists q\in \Q^{+})(D_{f, q}\ne \emptyset)$.  
For the other direction, a closed set $C\subsetneq [0,1]$ is empty if and only if the usco function $\mathbb{1}_{C}$ is continuous.  
Regarding $\Phi_{15}$, consider usco $f$ and $D_{f, q}$ as in the previous paragraph.  Then $f\in BV$ implies that $D_{f, q}$ is finite for all $q\in \Q^{+}$.  
Use $\Phi_{5,c}$ to decide the latter fact and obtain an enumeration of $D_{f}$ if such there is.
The latter enumeration allows one to compute $V_{0}^{1}(f)$ (if it exists), as the supremum in \eqref{tomb} now runs over $\N$ and $\Q$. 
For the reversal, a closed set $C\subsetneq [0,1]$ is empty if and only if $\mathbb{1}_{C}\in BV$ and $V_{0}^{1}(\mathbb{1}_{C})=0$.

\smallskip

Fourth, to compute $\Phi_{16}$, use $\Phi_{3}$ to define $f_{n}:[0,1]\di \R$ as follows:
\be\label{conv}\textstyle
f_{n}(x):= \sup_{y\in [0,1]}(f(y)- n |x-y|  ),
\ee
which is continuous by definition.
Clearly, the pointwise limit of $(f_{n})_{n\in \N}$ is $f$, i.e.\ the latter is Baire 1.  To show that $\Phi_{16}$ computes $\Phi_{3}$, we recall that 
the supremum of a Baire 1 function is computable using $\exists^{2}$ by the proof of \cite{dagsamXIV}*{Theorem 2.9}.
The functional $\Phi_{17}$ is treated in the same way while combining $\Phi_{12}$ and $\Phi_{16}$, one computes $\Phi_{13}$, which had not been treated yet. 

\smallskip

Fifth, to show that $\Phi_{4}$ computes $\Phi_{18}$ and $\Phi_{19}$, let $x_{0}\in [0,1]$ be such that $f(x_{0})=\inf_{x\in [0,1]}f(x)$, i.e.\ $f$ attains its minimum at $x_{0}$.  
Fix $\eps>0$ and suppose \eqref{denkf} is false, i.e.\ there is $y\ne x_{0}$ such that $\eps|y-x_{0}|\leq f(x_{0})-f(y)$.  
This implies $f(x_{0})>f(y)$ and contradicts the minimum property of $x_{0}$.  Hence, $x_{0}$ is $\eps$-critical for all $\eps>0$.  
Now assume $\Phi_{18}$ and let $(x_{k})_{k\in \N}$ be a sequence in $[0,1]$ such that $x_{k}$ satisfies \eqref{denkf} for $\eps=\frac{1}{2^{k}}$, for any $k\in \N$.
Let $(y_{k})$ be a convergent sub-sequence, say with limit $y\in [0,1]$, all found using $\exists^{2}$.  We now show that $f(y)=\inf_{x\in [0,1]}f(x)$, i.e.\ $\Phi_{4}$ is obtained. 
To this end, suppose there is $z\in [0,1]$ such that $f(z)<f(y)$.  Since $f$ is lsco, there is $k_{0}\in \N$ such that $f(z)+\frac{1}{2^{k_{0}}}<f(y_{k})$ for $k\geq k_{0}$.    
Hence, there is $k\geq k_{0}$, such that $\frac{1}{2^{k}}|y_{k}-z|  \leq f(y_{k})-f(z)$.  Since $y_{k}$ is $\frac{1}{2^{k}}$-critical, we obtain $z=y_{k}$, a contradiction. 

\smallskip

Sixth, $\Phi_{20}$ provides an honest RM-code, which can be converted into a Baire 1 representation using $\exists^{2}$ by (the proof of) \cite{ekelhaft}*{Lemma 6.3}; we obtain $\Phi_{16}$ as required.
For the other direction, we note that 
\[\textstyle
(\forall x\in B(a, r))(f(x)\geq q)\asa [q\leq \inf_{x\in (a-r, a+r)}f(x)],
\]
where the infimum operator is provided by $\Phi_{3}$ (and the fact that $f$ is usco if and only if $-f$ is lsco).
Hence, the required (honest) RM-code is readily defined. 

\smallskip

Seventh, $\Phi_{21}$ computes a realiser for the Urysohn lemma as follows: let $C_{0}, C_{1}$ be disjoint subsets of $[0,1]$ and define $O_{1}:=[0,1]\setminus C_{1}$. 
Apply $\Phi_{21}$ to $f(x):= \mathbb{1}_{C_{0}}(x)$ and $g(x):= \mathbb{1}_{O_{1}}(x)$ and note that the resulting continuous $h:[0,1]\di \R$ is such that $x\in C_{i}\di h(x)=1-i$ for $x\in [0,1]$ and $i\in \{0,1\}$, i.e.\ as required for the Urysohn lemma.   
For the reversal, one observes that the proof of \cite{good}*{Theorem 1} goes through without modification. 

\smallskip

Eight, for a closed set $C\subset [0,1]$, $\Phi_{22}$ provides a Baire 1 representation of $\mathbb{1}_{C}$.  
By (the proof of) \cite{dagsamXIV}*{Theorem 2.9}, $\exists^{2}$ computes the supremum of Baire 1 functions.   
Hence, we can decide whether $C$ is non-empty, yielding $\Omega_{C}$. 
For the other direction, the proof of \cite{SVR}*{Theorem 11.6} shows that $\Phi_{3}$ computes $\Phi_{22}$.  
To compute $\Phi_{24}$, we have $f_{\odot}(x)=\sup_{y\in [x, 1]} f(y)$ for usco $f$ (see \cite{SVR}*{Excercise 1.G}), i.e.\ $\Phi_{3}$ computes the former functional.    
For the other direction, $\Phi_{24}$ allows us to compute $\sup_{x\in [p, q]}f(x)$ by applying the former to $F$ defined as follows: $F(f, x, q):= f(x)$ for $x\leq q$ and $f(q)$ otherwise.  For fixed $q$ and usco $f$, $\lambda x.F(f,x, q)$ is also usco.   To show that $\Phi_{23}$ computes $\Phi_{2}$, in case $C=\{ x\in [0,1]: g(x)=0 \}$ where $C$ is closed and $g$ continuous, one readily computes an RM-code.  
Indeed, $\exists^{2}$ suffices to compute the sup and inf of continuous functions (see \cite{kohlenbach2}*{\S3}).  To compute $\Phi_{23}$, let $C$ be closed and consider $d(x, C)$, which is computable in terms of $\exists^{2}$ and an RM-code for $C$.  Since the distance function is continuous, use $\exists^{2}$ to define $f_{C}(x):=\int_{0}^{x}d(y, C)dy$.  Clearly, this function is increasing and satisfies the other conditions for $\Phi_{23}$ in light of the fundamental theorem of calculus.  

\smallskip

Nineth, regarding $\Phi_{25}$ and $\Phi_{26}$, let $C\subset [0,1]$ be closed and define the usco function $\mathbb{2}_{C}:[0,1]\di \overline{\R}$ 
as follows: $\mathbb{2}_{C}(x):= +\infty$ in case $x\in C$ and $\mathbb{2}_{C}(x)=0$ otherwise.
Clearly, $C=\emptyset \asa (\exists N\in \N)(\forall x\in [0,1])(\mathbb{2}_{C}(x)\leq N) $, i.e.\ $\Phi_{25}$ computes $\Omega_{C}$.    
The same holds for $\Phi_{26}$ as Riemann integrability implies boundedness.  For the other direction, let $f:[0,1]\di \overline{\R}$ be usco and note that this function is finitely bounded above if and only if 
$(\exists N\in \N)( E_{N}=\emptyset)$ for the closed set $E_{n}:=\{ x\in [0,1]:  f(x)\geq n\}$.  Hence, $\Phi_{1}$ computes $\Phi_{25}$.  For $\Phi_{26}$, we note that Riemann integrability is equivalent to boundedness plus continuity almost everywhere.  The latter property is equivalent to the closed sets $D_{f, q}$ from \eqref{tago} having measure zero for all $q\in \Q^{+}$.  Now $\Phi_{2}(D_{f, q})$ yields an RM-code for these sets and $\exists^{2} $ can decided whether the codes have measure zero.  Hence, $\Phi_{26}$ is also computable from $\Omega_{C}$.  

\smallskip

Tenth, to compute $\Phi_{27}$ from $\Phi_{8}$, let $C\subset [0,1]$ be closed and use $\Phi_{8}$ to obtain $\sup C$ (and $\inf C$).  In case $\inf C<_{\R} \sup C$, the set $C$ equals the interval $[\inf C, \sup C]$ if and only if $(\forall q\in \Q\cap [\inf C, \sup C])(q\in C)$.  Now assume $\Phi_{27}$ and consider an open set $O\subset[0,1]$ where $C:=[0,1]\setminus O$; define the closed sets $L_{q}:= C\cap [0, q]$ and $R_{q}:= C\cap [q, 1]$ and note that $O$ is an open interval if and only if $(\forall q\in \Q\cap O)(\textup{$L_{q}$ and $R_{q}$ are intervals}  )$.  Hence, $\Phi_{27}$ allows us to decide whether an open set $O\subset[0,1]$ is an interval.  
Since $(p, q)\subset O$ if and only if $(p, q)\cap O$ is an interval, $\Phi_{2}$ can be computed, as required. 

\smallskip

Eleventh, to compute $\Phi_{28}$, the proof of \cite{flinkenstein}*{Theorem 1} guarantees that given lsco $f:[0,1]\di [0, +\infty)$, the function $h:([0,1]\times \R)\di [0, +\infty)$ is definable (using only $\exists^{2}$) from an increasing sequence $(f_{n})_{n\in \N}$ of continuous functions with pointwise limit $f$; this sequence is obtained from $\Phi_{16}$.  The other direction is also immediate from the aforementioned proof.  

\smallskip

Finally, the combination $\Omega_C+\exists^2$ computes an RM-code of an open set, from which $\exists^{2}$ readily computes a Lebesgue number as required for $\Phi_{29}$. 
For the other direction, let $X \subseteq [0,1]$ be closed and let $\Phi_{29}$ be given. We will use $\exists^2$ and $\Phi_{29}$ to decide if $X$ is empty or not. 
Let $\big((p_n,q_n)\big)_{n \in \N}$ be an enumeration of all pairs of rational numbers $p < q$ in $ [0,1]$ and  let $O_n = (p_n,q_n) \cap [0,1]$ be the associated open covering. 
Further, if $[a,b] \subseteq [0,1]$ we write `$x \in [a,b]^{ X}$' if $\frac{x - a}{b-a} \in X$. Intuitively, $[a,b]^X$ is the range of $X$ under the affine map from $[0,1]$ onto $[a,b]$.

\smallskip

Now fix $\delta_{0} := \Phi(\lambda n.O_n)$ and let us construct a new covering $(O'_n)_{n \in \N}$ from the given one, depending on $X$ and $\delta_{0}$.    
Our construction will guarantee that each $O'_n = O_n$ if $X = \emptyset$ while no $O'_n$ will contain an interval of length $\delta_{0}$ if $X \neq \emptyset$. 
We can then test if $X$ is empty or not by testing if $\Phi(\lambda n.O'_n) = \delta_{0}$ or not.
 
 \medskip

For the construction, if $q_n - p_n < \delta_{0}$, we put $O'_n = O_n$, which ensures that our new sequence of sets will be a covering of $[0,1]$. 
If $q_n - p_n \geq \delta_{0}$, we let $k\in \N$ be such that $\frac{q_n - p_n}{k} < \delta_{0}$, and we put $a_{n,i}: = p_n + i(\frac{q_n - p_n}{k})$ for $i = 0 , \ldots ,k$. 
Finally, we define $O'_n$ from $O_n$ by removing all elements of the subsets of $O_n$ of the form $[a_{n,i},a_{n,i+1}]^{X} $ for $i = 0 , \ldots , k-1$. 
This new open covering clearly has the properties described in the previous paragraph, and we are done.
\end{proof}
Regarding $\Phi_{15}$, we could generalise to \emph{regulated} functions and output the associated \emph{Waterman} variation (see \cite{voordedorst}). 
Regarding $\Phi_{18}$ and $\Phi_{19}$, we could similarly study the Caristi fixed point theorem based on its second-order development in \cite{ekelhaft2}.
Regarding $\Phi_{21}$, we could obtain similar equivalences for the insertion theorems by Dowker and Michael (\cites{michael1, dowker1}), following the proofs in \cite{good}. 
We note that Hahn was the first to prove the above insertion theorem for metric spaces (\cite{hahn1}).  
Regarding $\Phi_{26}$, we note that the Lebesgue integral is not suitable here.  
Finally, we could also generalise certain results to Baire 1$^{*}$ as the latter amounts to the sum of an usco and a lsco function by \cite{mentoch}.

\subsection{More on the $\Omega_{C}$-cluster}\label{relo}
We establish further results related to the $\Omega_{C}$-cluster as follows.  
\begin{itemize}
\item We identify functionals in the $\Omega_{C}$-cluster that decide basic properties of \emph{arbitrary} functions (Section \ref{trini}).  
\item We generalise the notion of \emph{modulus of continuity} and identify functionals in the $\Omega_{C}$-cluster that output these generalisations (Section \ref{mocosec}).  
\item In Section \ref{trinix}, we (briefly) study computational equivalences for $\Omega_{C}$ (and $\Omega$) assuming functionals stronger than $\exists^{2}$.  
\end{itemize}
Most of the results in Sections \ref{trini}-\ref{trinix} could be included in the cluster theorems of Section \ref{clute}; nonetheless, we believe the former deserve their own section, due to their novelty but also in the interest of clarity of presentation.  

\subsubsection{Computing with arbitrary functions}\label{trini}
To decide basic properties of \emph{arbitrary} functions $f:[0,1]\di \R$, like boundedness or Riemann integrability, one needs (exactly) Kleene's $\exists^{3}$.  
In this section, we show that additionally assuming an oscillation function as input, the decision procedure is more tame, namely $\Omega_{C}$.  
\begin{thm}\label{ICT}
The following functionals belong to the $\Omega_{C}$-cluster:
\begin{itemize}
\item  for input any $f:[0,1]\di \R$ with oscillation function, decide if $f$ is continuous or not on $[0,1]$.  
\item  for input any $f:[0,1]\di \R$ with oscillation function, decide if $f$ is continuous or not, and output $x\in D_{f}$ in the latter case and a modulus of continuity in the former.
\item  for input any $f:[0,1]\di \R$ with oscillation function, decide if $f$ is in $BV$ or not on $[0,1]$, and output $V_{0}^{1}(f)$ in the former case.  
\end{itemize}
\end{thm}
\begin{proof}
Fix $f:[0,1]\di \R$ and its oscillation function $\osc_{f}:[0,1]\di \R$.  
For the first item, assume $\Omega_{C}$ and note that by the associated cluster theorem, we can decide whether the closed set 
\be\label{beg}\textstyle
D_{k}=\{ x\in [0,1]:\osc_{f}(x)\geq \frac{1}{2^{k}} \}
\ee
is non-empty.  Then $f$ is continuous everywhere on $[0,1]$ if and only if $D_{f}=\cup_{k\in \N}D_{k}$ is empty, where the latter is decidable using $\Omega_{C}+\exists^{2}$.  
Alternatively, $f$ is continuous everywhere on $[0,1]$ if and only if we have that $\osc_{f}$ is continuous everywhere on $[0,1]$ \emph{and} $\osc_{f}(q)=0$ for all $q\in [0,1]\cap \Q$.
Since $\osc_{f}$ is usco, we can use $\Phi_{14}$ from Theorem \ref{seclu} to decide the continuity of $f$.
Now assume a functional as in the first item and consider $C\subset [0,1]$.  We may assume that $C\cap \Q=\emptyset$ as $\mu^{2}$ can enumerate the rationals in $C$.  
Define $f(x):= \mathbb{1}_{C}$ and note that $\osc_{f}(x)=f(x)$ for all $ x\in [0,1]$ by a straightforward case distinction. 
Thus, the first centred functional can decide whether $f$ is continuous on $[0,1]$, which is equivalent to deciding whether $C=\emptyset$, i.e.\ we obtain $\Omega_{C}$ via the associated cluster theorem. 

\smallskip

For the second item, in case $D_{k}\ne \emptyset$, the usual interval-halving technique will provide $x\in D_{k}\subset D_{f}$.  
A modulus of continuity for a continuous function is computable in $\exists^{2}$ by \cite{kohlenbach4}*{\S4}.

\smallskip

For the third item, assume $\Omega_{C}$ and consider again the closed set $D_{k}$ from \eqref{beg}.
Using $\Phi_{5, c}$ from Theorem \ref{ficlu}, we can decide whether $D_{k}$ is finite, and if so, enumerate it.  
Now, assuming $D_{f}$ is countable and given an enumeration of this set, $\exists^{2}$ can compute $V_{0}^{1}(f)$ (if it exists) as the supremum in \eqref{tomb} now runs over $\Q\times \N$. 
Hence, we obtain the third centred functional.  
To show that the third centred functional computes $\Omega_{C}$, note that for a closed set $C\subset [0,1]$, one has $C=\emptyset $ if and only if $\mathbb{1}_{C}$ is in $BV$ with $V_{0}^{1}(\mathbb{1}_{C})=0$.  
Moreover, since $\Omega_{C}$ computes $\Omega$, we can obtain the variation function $V_{0}^{x}(f)$ for any $f\in BV$ by the second cluster theorem for $\Omega$ (\cite{dagsamXII}). 
The oscillation function is obtained in the same way as in the first paragraph.   
\end{proof}
Regarding the final item of Theorem \ref{ICT}, we could generalise to \emph{regulated} functions and output the associated \emph{Waterman} variation (see \cite{voordedorst}). 

\subsubsection{On moduli of continuity}\label{mocosec}
The notion of \emph{modulus of continuity} is central to a number of computational approaches to mathematics.  In this section, we generalise this concept and establish a connection to the $\Omega_{C}$-cluster.  

\smallskip

First of all, we need the following definitions, where we always assume that the functions at hand are not totally discontinuous, i.e.\ $C_{f}\ne \emptyset$.  
%
\bdefi
For $f:[0,1]\di \R$, a function $F:(\R\times \N)\di \N$ is a \emph{modulus of continuity} if 
\be\label{tok1}\textstyle
(\forall k\in \N, x\in C_{f},y\in [0,1] ) ( |x-y|<\frac{1}{2^{F(x, k))}}\di |f(x)-f(y)|<\frac{1}{2^{k}}).
\ee
\edefi
\begin{defi}
Let $M_{\Gamma}:(\R\di \R)\di (\R\times \N)\di \N$ be a functional that on input $f:[0, 1]\di \R$ in the function class $\Gamma$ ouputs a modulus of continuity.  
\edefi
Secondly, we have the following where `$\BV$' stands for `bounded variation'.
\begin{thm}\label{thm2.9}
Assuming $\exists^{2}$, $\Theta + M_{\BV}$ computes $\Omega$ and $\Omega$ computes $M_{\BV}$.
\end{thm}
\begin{proof}
For the first part, to compute $\Omega_{ b}$, let $X \subset [0,1]$ have at most one element, and let $f$ be the characteristic function of $X$. Trivally, $f$ is in ${BV}$, and let $F:(\R\times \N)\di \N$ be a modulus of continuity of $f$.  Define $\Psi(x) =\frac{1}{2^{ F(x,1)}}$ if $x \not \in X$ and $\Psi(x) =1$ if $x \in X$ and consider $\Theta(\Psi) = \{x_1 , \ldots , x_n\}$.  By definition, $X$ is non-empty if and only if  one of the $x_i$ is in $X$.

\smallskip

For the second part, we have $x\in C_{f}\asa [f(x+)=f(x)=f(x-)]$ and $\exists^{2}$ can compute left and right limits for regulated functions.  
Moreover, $\Omega$ can enumerate $D_{f}$ by the second cluster theorem from \cite{dagsamXII}.  
A modulus of continuity is then readily defined by restricting the quantifier over the reals in `$x\in C_{f}$'  to: $\Q$ and the enumeration of $D_{f}$. 
\end{proof}
Thirdly, we can improve the previous theorem by considering the functional $\Delta$ from \cite{dagsamVII}*{\S6-7}.  Recalling Definition \ref{char}, this functional is such that for R2-open $O\subset [0,1]$, $\Delta(O)+\exists^{2}$ computes an RM-code for $O$.   
\begin{thm}
Assuming $\exists^{2}$, $\Delta + M_{\BV}$ computes $\Omega$.  
\end{thm}
\begin{proof}We let $X$, $f$ and $F$ be as in the first part of the proof of Theorem \ref{thm2.9}, and we let $\Psi_1$ be defined like $\Psi$ in that proof with the exception that $\Psi_1(x) = 0$ when $x \in X$. Let $O = [0,1] \setminus X$. Then $\Psi_1$  is an R2-representation of $O$, and $\Delta$ provides an RM-code for $O$. Using $\exists^2$ we can now decide $X=\emptyset$, as required for $\Omega_{b}$.
\end{proof}
Finally, we establish a connection to $\Omega_{C}$ as follows.  
\begin{thm}
The combination $\exists^{2}$, $M_{\usco}+\Delta$ is in the $\Omega_{C}$-cluster.  
\end{thm}
\begin{proof}
First of all, let $C\subset [0,1]$ be closed and let $F_{C}:(\R\times \N)\di \N$ be a modulus of continuity for the usco function $\mathbb{1}_{C}$. 
Then $x_{0}\in O:= [0,1]\setminus C$ implies $x_{0}\in C_{\mathbb{1}_{C}}$ by definition. Moreover, $B(x_{0}, \frac{1}{2^{F(x_{0}, 2)}})\subset O$, as required for an R2-representation of $O$.  
Then $\Delta$ provides an RM-code for $C$, yielding $\Omega_{C}$ by the cluster theorems.  

\smallskip

Secondly, $\exists^{2}$ suffices to decide `$x\in C_{f}$' for usco functions (\cite{dagsamXIII}*{\S2} or \cite{samBIG2}*{Theorem 2.4}.), i.e.\ we can provide a trivial output for $x\in D_{f}$ and the following for $x\in C_{f}$:  
since $\Omega_{C}$ can compute the sup and inf of an usco function over a given interval by the cluster theorems, one readily defines $M_{\usco}$.
Since $\Omega_{C}$ computes RM-codes for \emph{any} closed set, $\Delta$ is trivially obtained. 
\end{proof}
Finally, we note that for any class $\Gamma$ of functions, if one can compute $\Omega_b$ from the continuity problem for $\Gamma$, and each $\Pi^0_2$-set can be represented as $D_f$ for some $f \in \Gamma$, then the continuity problem for $\Gamma$ is as hard as $ \exists^3$ (assuming \textsf{V=L}).

\subsubsection{A word on the Suslin functional}\label{trinix}
We sketch some computational equivalences for $\Omega_{C}$ assuming the Suslin functional instead of Kleene's quantifier $\exists^{2}$.  

\smallskip

First of all, we have established the cluster theorems for $\Omega_{C}$ in Section \ref{clute}, where we always assume $\exists^{2}$ in the background.
To motivate (the innocence of) this assumption, we point out that by Theorem \ref{tame}, $\Omega_{C}+\exists^{2}$ computes the same real numbers as $\exists^{2}$ alone.  
By contrast (and Footnote \ref{labour}), $\Omega_{C}+\SS^{2}$ computes $\SS_{2}^{2}$, and even $\exists^{3}$ if \textsf{V=L}. 
Thus, the Suslin functional is rather `explosive' when combined with $\Omega_{C}$ and thus not an innocent background assumption.

\smallskip

Secondly, despite the previous `warning paragraph', we have the following theorem which implies that functionals satisfying \eqref{lekal} are computationally equivalent to $\Omega_{C}$ \textbf{assuming} $\SS^{2}$.
\begin{thm}\label{pringke}
The combination $\Omega_{C}+\SS^{2}$ computes a functional with specification:
\be\label{lekal}
\textup{on input lsco $f:[0,1]\di \R$, decide if $f$ is continuous on $[0,1]$ or not.}
\ee
Given $\exists^{2}$, any functional satisfying \eqref{lekal} computes $\Omega_{C}$.  
\end{thm}
\begin{proof}
Fix lsco $f:[0,1]\di \R$ and recall that $x\in C_{f}$ is decidable using $\exists^{2}$ by \cite{dagsamXIII}*{\S2} or \cite{samBIG2}*{Theorem 2.4}.
In particular, in case $x\in D_{f}$, then $f$ is not usco at $x$, i.e.\ 
\be\label{ting}\textstyle
(\exists l\in \N)(\forall N\in \N)(\exists z\in B(x, \frac{1}{2^{N}}))(f(z)\geq f(x)+\frac{1}{2^{l}}).
\ee
Note that we used the fact that $C_{f}$ consists of those points where $f$ is lsco and usco simultaneously.  
Since $f$ is lsco everywhere, \eqref{ting} is equivalent to:
\be\label{ting2}\textstyle
(\exists l\in \N)(\forall N\in \N)\underline{(\exists r\in B(x, \frac{1}{2^{N}})\cap \Q)}(f(r)\geq f(x)+\frac{1}{2^{l}}),
\ee
where the underlined quantifier is essential.   Now, `$x\in D_{f}$' is equivalent to \eqref{ting2} and $\Omega_{C}$ allows us to replace 
$f$ in \eqref{ting2} by second-order codes by the cluster theorems.  Hence, $(\exists x\in [0,1])(x\in D_{f})$ is decidable using $\SS^{2}$. 
The second part is immediate by the above cluster theorems.  
\end{proof}
Finally, the interested reader may verify that the following functionals are computationally equivalent to $\Omega_{C}$ assuming $\SS^{2}$.  
\begin{itemize}
\item Any functional that on input usco $f:[0,1]\di \R$, decides if $f$ is continuous on $[0,1]$ or not. 
\item Any functional that on input usco $f:[0,1]\di \R$, decides if $f$ is differentiable on $(0,1)$ or not. 
\item Any functional that on input usco $f:[0,1]\di \R$, decides if $f$ is Lipschitz on $[0,1]$ or not.  
\item Any functional that on input usco $f:[0,1]\di \R$ and $\alpha>0$, decides if $f$ is $\alpha$-H\"older-continuous on $[0,1]$ or not.  
\item Any functional that on input usco $f:[0,1]\di \R$, decides if $f$ is absolutely continuous on $[0,1]$ or not.  
\item Any functional that on input usco $f:[0,1]\di \R$, decides if $f$ is quasi-continuous on $[0,1]$ or not.
\end{itemize}
There are many function spaces between e.g.\ absolute continuity and quasi-continuity, many of which should yield similar results.  
We note that the characteristic function of closed sets is automatically cliquish, i.e.\ the latter provides an upper bound on the generality of what is computable by $\Omega_{C}+\SS^{2}$
\section{The computational landscape around $\Omega_{C}$}\label{dapa}
\subsection{Introduction}
In the previous section, we have identified numerous functionals stemming from mainstream mathematics that are computationally equivalent to $\Omega_{C}$ assuming $\exists^{2}$, giving rise to the $\Omega_{C}$-cluster.  
In this section, we investigate the connections between $\Omega_{C}$ and known functionals, an overview of which may be found in Figure \ref{xxx}. 
In particular, we establish the following results.  
\begin{itemize}
\item We show that $\Omega_{C}$ is \emph{lame} in Theorem \ref{tame}, i.e.\ the combination $\Omega_{C}+\exists^{2}$ computes the same real numbers as $\exists^{2}$.
\item We show that there is no total functional of type three intermediate\footnote{To be absolutely clear, Theorem \ref{hench} expresses that there is no total $\Phi^{3}$ such that: $\Omega_{C}+\exists^{2}$ computes $\Phi$ and $\Phi+\exists^{2}$ computes $\Omega$.} between $\Omega$ and $\Omega_{C}$ from the computational point of view (Theorem \ref{hench}).  
\item In contrast to the previous item, we identify a number of natural \emph{partial} functionals in Section \ref{betwi} that are intermediate between $\Omega$ and $\Omega_{C}$.
\item In Section \ref{CIT}, we study the computational properties of the \emph{Cantor intersection theorem} as the latter yields witnessing functionals intermediate between $\Omega_{C}$ and weak Cantor realisers.  
\item In Section \ref{LM}, we study the computational properties of the Lebesgue measure (restricted to closed sets) as it relates to $\Omega$ and $\Omega_{C}$.
\end{itemize}
Regarding the final item, it turns out that that the Lebesgue measure has a cluster theorem (Theorem \ref{lam}) similar to the cluster theorems for $\Omega_{C}$ from Section \ref{clute}, 
which we found somewhat surprising. 

\subsection{On the lameness of $\Omega_{C}$}\label{lameoc}
We show that the combination $\Omega_{C}+\exists^{2}$ does not compute more real numbers than $\exists^{2}$ alone.  
As an important consequence, $\exists^3$ is not computable in $\Omega_{C}+\exists^2$

\smallskip

First of all, we have introduced the following notion in \cite{dagsamXIII} in relation to $\Omega$. 
\bdefi\label{XGG}
A (partial) functional $\Phi^{3}$ is \emph{lame} if for all $x$ and $y$ of type 1, if $y$ is computable in $x+\Phi+\exists^2$, then $y$ is computable in $x+\exists^2$.
\edefi
We point out that $\Omega_{C}+\SS^{2}$ computes\footnote{We have that $\Omega_{C}$ computes $\Omega$ assuming $\exists^{2}$, and $\Omega+\SS^{2}$ computes $\SS_{2}^{2}$ by \cite{dagsamX}*{Theorem 4.3}.  The analogous result for $\exists^{3}$ is \cite{dagsamX}*{Theorem 4.6}. \label{labour}} $\SS_{2}^{2}$ and even $\exists^{3}$ if $\textsf{V=L}$, i.e.\ the former combination \emph{is} much stronger than $\SS^{2}$ alone.

\smallskip

Secondly, we need the following results where item (a) in Lemma \ref{Alain} is a special case of the Louveau separation theorem (\cite{Sacks.high}*{Corollary IV.6.2}). For the sake of completeness, we provide a full proof.
\begin{lemma}\label{Alain}~
\begin{itemize}
\item[(a)] If a hyperarithmetical set $X \subseteq \N^\N$ is closed, then $X$ has a hyperarithmetical RM-code as a closed set. 
\item[(b)] Further, if a closed set $X \subseteq 2^\N$ is non-empty and with a hyperarithmetical code, then $X$ contains a hyperarithmetical element.
\end{itemize}
\end{lemma}
\begin{proof} Item (b) is trivial, since from the code we can define the least element of $X$,  in the lexicographical ordering, by an arithmetical formula.  To prove (a), 
let $X \subseteq \N^\N$ be closed and hyperarithmetical. Then
\[
(\forall f^{1} )(\exists n^{0})( \forall g^{1}) (\overline f(n) = \overline g(n) \rightarrow g \not \in X \vee f \in X)
\]
By $\Pi^1_1$-uniformisation over $\N$, there is a hyperarithmetical function $F$ such that
\[
(\forall f^{1}  ,  g^{1}) (\overline f(F(f)) = \overline g(F(f)) \rightarrow g \not \in X \vee f \in X)
\]
We now consider two disjoint $\Sigma^1_1$-sets of sequence numbers as follows:
\[
A:= \{s : (\exists f^{1}) (f \not \in X \wedge s = \overline f(F(f)))\} \textup{ and } B:= \{t : (\exists f^{1}) (f \in X \wedge t \prec f)\}.
\]
By $\Sigma^1_1$-separation over $\N$, there is a \emph{hyperarithmetical} set $C$ of sequence numbers containing $A$ and disjoint from $B$.  This set $C$ then defines a code for $X$, since the complement of $X$ is the union of the neighbourhoods defined from the elements of $C$, as required for item (a).
 \end{proof}
We have proved item (a) in Lemma \ref{Alain} for $\N^\N$, but we cannot prove item (b) for $\N^\N$, mainly because it is false.

\smallskip
\noindent
Thirdly, we can now establish that $\Omega_{C}$ is weak in the sense of Def.\ \ref{XGG}.
\begin{theorem}\label{tame}
The functional $\Omega_{C}$ is lame.
\end{theorem}
\begin{proof}
The proof is based on the relativised version of Lemma \ref{Alain} and the recursion theorem.  Indeed, we observe that whenever we have a terminating term $t[\Omega_{C}, \exists^2,\vec x]$, where $\vec x$ is a sequence of objects of type 0 or 1, we can replace this by an equivalent term $t^*[\exists^2,\vec x]$. As proved in \cite{Pla}, he recursion theorem is valid in our setting, and we use Theorem \ref{thm5.7} to show that $t^*[\exists^2,\vec x]$ terminates whenever $t[\Omega_{C}, \exists^2,\vec x]$ terminates.

\smallskip

The only case where we need to do something is the application of $\Omega_{C}$ to a characteristic function $
\lambda z^{1}.s[\Omega_{C},\exists^2,\vec x]$; by the assumption of termination this is the characteristic function of a compact set $X \subseteq 2^\N$. 
By the induction hypothesis, we have that $\lambda z^{1}.s^*[\exists^2,\vec x]$ is (also) the characteristic function of $X$, so $X$ is hyperarithmetical in $\vec x$. 
By Lemma \ref{Alain} we have that 
\[
X = \emptyset  \leftrightarrow \big(\forall z \in \textsf{HYP}(\vec x)\big)(z \not \in X), 
\] 
which is clearly $\Delta_{1}^{1}$. 
We can thus use $\exists^2$ to decide if $X$ is empty or not.
\end{proof}
We now discuss a (slight) generalisation $\Omega_{C}$ which is no longer lame. 
Generalisations and variations of $\Omega_{C}$ are discussed in detail in Section \ref{hopen}.
\begin{remark}[From Cantor to Baire space]\rm
In contrast to the lameness of $\Omega_{C}$, the partial functional $\Phi$ that on input any closed $X\subset \N^\N$, decides $X=\emptyset$, is much stronger.  
Indeed, $\Phi$ computes \emph{both} $\Omega_{b}$ and the Suslin functional $\SS^2$, and thus $\SS^{2}_{2}$ by \cite{dagsamX}*{Theorem 4.3}.  Assuming $\textsf{V=L}$, $\Phi$ even computes $\exists^3$ by \cite{dagsamX}*{Theorem 4.6}. For the same reason, testing the emptiness of a $G_\delta$-subset of $2^\N$ is too strong to be of interest. As we are moving into independence results, we will not pursue this aspect further in this paper.
\end{remark} 
In \cite{dagsamXIII} we proved that $\Omega_{b}$ is not computationally equivalent, modulo $\exists^2$, to any total functional of type 3. The same will hold for $\Omega_{C}$, with only minor modifications to the proof. The modifications actually give us a slightly stronger theorem, a theorem that we prove just by pointing to the modifications needed in the proof of \cite{dagsamXIII}*{Theorem 3.46}.
\begin{thm}\label{hench}
There is no total functional $\Phi$ of type 3 such that
\begin{itemize}
\item $\Phi$ is computable in $\Omega_{C}$ and $\exists^2$,
\item $\Omega_{b}$ is computable in $\Phi$ and $\exists^2$.
\end{itemize}
\end{thm}
\begin{proof}
We follow and modify the steps in the proof of \cite{dagsamXIII}*{Theorem 3.46} as follows. 
\begin{enumerate}
\item We reformulate \cite{dagsamXIII}*{Lemma 3.42} writing $\Omega_{C}$ instead of $\Omega_{b}$. 
The proof remains the same, where we at the crucial step use that a non-empty hyperarithmetical closed subset of $2^\N$ contains a hyperarithmetical element instead of the same for singletons.
\item We reformulate \cite{dagsamXIII}*{Theorem 3.43} to:
\begin{center}
{\em there is no total extension of $\Omega_{b}$ that is computable in $\Omega_{C} + \exists^2$.} 
\end{center}
The proof of the reformulation remains the same.
\item In \cite{dagsamXIII}*{Lemma 3.44} we replace $\Omega_{b}$ with $\Omega_{C}$. The proof remains the same, given the proofs of the lameness of $\Omega_{b}$ and $\Omega_{C}$. 
\item The next step, \cite{dagsamXIII}*{Lemma 3.45}, remains unchanged.
\item Now we can prove our theorem using the proof of \cite{dagsamXIII}*{Theorem 3.46}, using \cite{dagsamXIII}*{Theorem 3.41} in the same way as in \cite{dagsamXIII}.
\end{enumerate}
Together with the proof of \cite{dagsamXIII}*{Theorem 3.46}, this proof is now complete.  
\end{proof}

\subsection{Between $\Omega$ and $\Omega_{C}$}\label{betwi}
In this section, we discuss some natural functionals that exist between $\Omega$ and $\Omega_{C}$, based on mainstream notions like Riemann integration.

\smallskip

First of all, to decide basic properties of a given function $f:[0,1]\di \R$, like boundedness or Riemann integrability, one needs (exactly) Kleene's $\exists^{3}$.  
We now show that assuming an oscillation function as input, the decision procedure is more tame, namely at most $\Omega_{C}$.  
We allow the special value `$+\infty$' for the oscillation function and supremum in the following theorem.  
\begin{thm}\label{kahi}
The following are computable in $\Omega_{C}+\exists^{2}$:
\begin{itemize}
\item any functional that on input $f:[0,1]\di \R$ and its oscillation function, decides if $f$ is Riemann integrable on $[0,1]$ or not, 
\item any functional that on input $f:[0,1]\di \R$ and its oscillation function, decides if $f$ is bounded on $[0,1]$ or not, 
\item any functional that on input $f:[0,1]\di \R$ and its oscillation function, outputs the supremum of $|f|$.
\end{itemize}
Given $\exists^{2}$, any of these functionals computes $\Omega$.  
\end{thm}
\begin{proof}
To show that $\Omega_C$ computes the functional in the first item, we consider the following well-known facts from real analysis for $f:[0,1]\di \R$ and its oscillation function $\osc_{f}:[0,1]\di \R$.
\begin{itemize}
\item  $f$ is bounded on $[0,1]$ \emph{if and only} if for large enough $q\in \Q^{+}$, the closed set $E_q := \{x \in [0,1] : \osc_f(x) \geq q\}$ is empty.  
\item $f$ is continuous almost everywhere on $[0,1]$ \emph{if and only if} for all $k\in \N$, the closed set $D_k = \{ x \in [0,1]: \osc_f(x)\geq 1/2^k \}$ has measure zero.  
\end{itemize}
Since $\Omega_{C}$ can decide whether closed sets are non-empty and/or have measure zero\footnote{By the cluster theorems, $\Omega_{C}$ computes an RM-code for a closed set $C\subset [0,1]$.  Then $\exists^{2}$ can decide if this code completely covers $[0,1]$, as this is (equivalent to) an arithmetical statement.  Similarly, given the code, deciding whether the associated set has measure zero, is arithmetical.}, we are done.
Indeed, Riemann integrability is equivalent to boundedness plus continuity ae, due to the Vitali-Lebesgue theorem.  
The supremum can be found using the usual interval-halving technique.  

\smallskip

To show that the functional in the first item computes $\Omega_b$, let $X\subset \R$ have at most one element.   We now define $g_X:[0,1]\di \R$ which is such that  
\begin{center}
$X= \emptyset$ \emph{if and only if} $g_{X}$ is Riemann integrable on $[0,1]$.
\end{center}  
Indeed, define $g_X:[0,1]\di \R$ and its oscillation function as follows: 
\be\label{geex}
g_{X}(x):=
\begin{cases}
2^{n} & \textup{if $n\in \N$ is such that $x\pm\frac{1}{2^{n}}\in X$}\\
0 & \textup{otherwise}
\end{cases},
\ee
where the oscillation function $\osc_{g_X}$ is defined as follows:
\be\label{geex2}
\osc_{g_{X}}(x):=
\begin{cases}
+\infty & x\in X\\
2^{n} & \textup{if $n\in \N$ is such that $x\pm\frac{1}{2^{n}}\in X$}\\
0 & \textup{otherwise}
\end{cases},
\ee
These functions are as required and we are done with the first item.  The other items are treated in exactly the same way. 
\end{proof}
We note that the use of the Riemann integral is essential: the function $g_{X}$ from \eqref{geex} is always Lebesgue integrable in case $|X|\leq 1$.  
It is a natural question what is needed to obtain $\Omega_{C}$ from the itemised functionals in Theorem \ref{kahi}.  In this regard, it is an interesting exercise to show that the functional in the third item of Theorem~\ref{kahi} computes $\Omega_{C}$ restricted to \emph{nowhere dense} closed sets.  

\smallskip

We finish this section with a remark on $\Omega_{b}$. 
\begin{rem}[The $\Omega_{b}$-cluster]\rm
We have obtained many computational equivalences involving $\Omega_{b}$ in \cite{dagsamXIII}.  
Below, we list an addition to the $\Omega_{b}$-cluster as it carries both historical and contemporary interest. 

\smallskip

First of all, a function $H:\R\di \N$ is a \emph{height \(function\)} for $A\subset \R$ in case $A_{n}:= \{ x\in [0,1]: H(x)<n\}$ is finite for all $n\in \N$.  
As it turns out, height functions are essential for the development of the RM of the uncountability of $\R$ (\cite{samBIG}) and the Jordan decomposition theorem (\cite{dagsamX}).  
Height functions can be found in the modern literature, but also go back to Borel and Drach circa 1895 (see \cites{opborrelen3, opborrelen4, opborrelen5})

\smallskip

Secondly, consider the following functional that ouputs Borel's $\varphi$-function from \cite{opborrelen4} where Borel computes $\varphi$ for various countable sets, like the algebraic numbers.  
\begin{center}
\emph{Any $\Phi:(\R\di \R)^{2}\di (\N^{2}\di \R)$ that on input $A\subset \R$ and a height $H:\R\di \N$ for $A$, outputs $\Phi(A, H)=\varphi$ such that $(\forall x, y\in A)( |x-y|>\varphi(H(x), H(y)))$. }
\end{center}
As an exercise, the reader can verify that $\Omega_{b}+\exists^{2}$ computes such a functional, while any such functional computes $\Omega_{b}$, assuming $\exists^{2}$. 
\end{rem}

\subsection{The Cantor intersection theorem}\label{CIT}
In this section, we study the computational properties of the Cantor intersection theorem, especially in relation to $\Omega_{C}$. 

\smallskip

First of all, the following specification formalises the notion of witnessing functionals for the Cantor intersection theorem.  
\bdefi\label{keffer2}
Any functional $\Phi:((\N\times \R)\di \N)\di \R$ is a \emph{Cantor intersection functional} if for any sequence $(C_{n})_{n\in \N}$ of non-empty closed sets with $C_{n+1}\subseteq C_{n}\subseteq [0,1]$ for all $n\in \N$, 
we have $\Phi(\lambda n. C_{n})\in \cap_{n\in \N}C_{n}$.
\edefi
Secondly, we establish the implications in Figure \ref{xxx}, and some equivalences. 
\begin{thm}
Assuming $\exists^{2}$, the following are computationally equivalent.
\begin{itemize}
\item Any selection functional $\Omega_{C, w}$ for closed \textbf{non-empty} sets in $[0,1]$.
\item Any functional $\Phi_{1, w}$ which on input usco $f:[0,1]\di \R$ \textbf{and} the real $y=\sup_{x\in [0,1]}f(x)$, outputs $x\in [0,1]$ such that $f(x)=y$.
\item Any Cantor intersection functional.  
\end{itemize}
Any Cantor intersection functional computes a weak Cantor realiser, assuming $\exists^{2}$. 
\end{thm}
\begin{proof}
First of all, that $\Omega_{C, w}$ computes $\Phi_{1, w}$ follows by considering the closed and non-empty set $\{x\in [0,1]:f(x)=y\} $ where $f$ is usco with supremum $y$.  
For the reversal, let $C\subset [0,1]$ be non-empty and closed.  Then $\Phi_{1, w}( \lambda x.\mathbb{1}_{C}(x), 1 )$ is an element of $C$ as characteristic functions of closed sets are usco. 

\smallskip

Secondly, to compute a Cantor intersection functional, let $(C_{n})_{n\in \N}$ be a sequence of non-increasing, closed, and non-empty sets in $[0,1]$.   Then $\Omega_{C, w}(C_{n})\in C_{n}$ for all $n\in \N$ by definition.  
Then $\exists^{2}$ computes a convergent sub-sequence with limit $z$ from the sequence $(\Omega_{C, w}(C_{n}))_{n\in \N}$; clearly, we have $z\in \cap_{n\in \N}C_{n}$, as required.  
That any Cantor intersection functional computes $\Omega_{C, w}$ is immediate by taking a constant sequence of the same closed non-empty set.  
The final sentence follows by considering the non-empty closed sets $E_{n}=\{x\in A: Y(x)=n \}$ where $Y$ is injective and surjective on $A\subset [0,1]$.  This yields an enumeration of $A$, and $y\not\in A$ is computed via the usual (computable) diagonal argument. 
\end{proof}
In contrast to $\Omega_{b}$, it is fairly easy to show that certain Cantor intersection functionals are strictly weaker than $\Omega_{C}$, following Theorem \ref{lef}.
The proof of the latter assumes Cantor's continuum hypothesis $\CH$, which was also assumed for results in \cite{dagsamIX}.  An important notion is \emph{countably based}, which we have studied in e.g.\ \cites{dagsamXIII}.  
\bdefi
Let the partial functional $\Phi^{\sigma}$ be of type $\sigma = \tau_1 , \ldots , \tau_n \rightarrow 0$, where $\tau_i = \delta_{i,1} , \ldots , \delta_{i,m_i} \rightarrow 0$ and each $\delta_{i,j}$ has rank $1$.  
We say that $\Phi$ is \emph{countably based} if, whenever $\Phi(F_1 , \ldots , F_n) \in \N$, there are countable subsets $X_i $ containing tuples of type $ \delta_{i,1} \times \cdots \times\delta_{i,m_i}$ such that $\Phi(F_1 , \ldots , F_n) = \Phi(G_1 , \ldots , G_n)$ whenever $\Phi(G_1 , \ldots , G_n) \in \N$ and  $F_i$ and $G_i$ are equal on $X_i$ for each $i = 1 , \ldots , n$.
\end{defi}
The \emph{total} countably based functionals were originally suggested by Stan Wainer and then studied by John Hartley in \cite{hartleycountable,hartjeS}.

\begin{thm}[$\CH$]\label{lef}
There exists a Cantor intersection functional $\Phi_{0}$ that is countably based.  
Hence, $\Phi_{0}+\exists^{2}$ cannot compute $\Omega_{C}$.
\end{thm}
\begin{proof}
Assume $\CH$ and let $\prec$ be a well-ordering of $[0,1]$ of order-type $\aleph_1$. For all non-empty sets $A \subseteq [0,1]$, let $\Phi_{0}(A)$ be the $\prec$-least element of $A$.
Then $\Phi_{0}$ is a countably based partial functional because the formula
\[
x \in A \wedge (\forall y\in [0,1])( y \prec x \di y \not \in A).
\]
guarantees that $x = \Phi_{0}(A)$ without knowing anything more about $A$.
By \cite{dagsamXIII}*{Theorem 3.27}, the countably based partial functionals are closed under Kleene computability.  
Since $\Omega_{C}$ is not countably based, $\Phi_{0}+\exists^{2}$ cannot compute it.
\end{proof}
Finally, we note that $\Omega_{C}+\exists^{2}$ computes a witnessing functional for the countable Heine-Borel theorem as follows, studied in \cite{dagsamVII} in a different formulation.
\begin{princ}
For any sequence $(O_{n})_{n\in \N}$ of open sets such that $[0,1]\subset \cup_{n\in \N}O_{n}$, there is $n_{0}\in \N$  such that $[0,1]\subset\cup_{n\leq n_{0}}O_{n}$. 
\end{princ}
We also note that $\blambda_{C}+\exists^{2}$ can compute a witnessing functional for the countable Vitali covering theorem, as follows.
\begin{princ}
For any sequence $(O_{n})_{n\in \N}$ of open sets with $[0,1]\subset \cup_{n\in \N}O_{n}$ and any $k\in \N$, there is $n_{0}\in \N$  such that $\cup_{n\leq n_{0}}O_{n}$ has total length at least $1-\frac{1}{2^{k}}$. 
\end{princ}
It would be interesting to know the exact relationship between these functionals. 

\subsection{The Lebesgue measure}\label{LM}
We show that the Lebesgue measure (restricted to closed sets) boasts many computational equivalences (Section \ref{euiw})
and is computationally weak in that the former does not compute $\Omega_{b}$ given $\exists^{2}$ (Section \ref{leak})

\subsubsection{A cluster theorem for the Lebesgue measure}\label{euiw}
We establish the cluster theorem for the Lebesgue measure (restricted to closed sets) as in Theorem \ref{lam}.
To this end, we need some definitions from measure theory, as follows.

\smallskip

First of all, the notion of \emph{essential supremum} $\textup{\textsf{ess-sup}}_{x\in X}f(x)$ is well-known from analysis (see e.g.\ \cite{sobo}*{\S2.10}).  Intuitively, this notion is the supremum where one is allowed to ignore a measure zero set.  
\bdefi
A real $y\in \R$ is the \emph{essential supremum} of $f:[0,1]\di \R$ in case $\{ x\in [0,1]: f(x)\geq y\}$ has measure zero and $\{ x\in [0,1]: f(x)\geq y-\frac{1}{2^{k}}\}$ has positive measure for all $k\in \N$.
\edefi
Secondly, the following is an equivalent definition of Jordan measurability (see e.g.\ \cites{frink, taomes}) where the classical one expresses that the Jordan inner and outer content are equal; this content amounts to the definition of the Lebesgue measure in which only \emph{finite} interval coverings are allowed. 
\bdefi 
A set $E\subset[0,1]$ is \emph{Jordan measurable} if the boundary $\partial E$ has Lebesgue measure zero. 
\edefi
Singletons are of course Jordan measurable but $\Q$ is not, i.e.\ the countable additivity property for Jordan content fails for elementary examples.  

\smallskip

Thirdly, we have the following theorem where we assume that the functionals in Theorem~\ref{lam} are undefined outside of their domain of definition.  
We have refrained from mixing lambda-abstraction and Lebesgue measure as in $\lambda C. \blambda(C)$, which we recommend the reader do as well.
\begin{thm}[Cluster theorem for the Lebesgue measure] \label{lam}
The following are computationally equivalent given $\exists^{2}$. 
\begin{itemize}
\item \($\blambda_{C}$\) The Lebesgue measure $\blambda:(\R\di \R)\di \R$ restricted to closed $C\subset [0,1]$.
\item The functional $L_{0}$ that for closed $C\subset [0,1]$ decides whether $\blambda(C)=0$.
\item The functional $L_{0, c}$ that for usco $f:[0,1]\di \R$ decides if $f$ is zero ae. 
\item The functional $L_{1}$ that to usco $f:[0,1]\di \R$ outputs $\textup{\textsf{ess-sup}}_{x\in [0,1]}f(x)$.  
\item The functional $L_{2}$ that to closed $C\subset [0,1]$, outputs an RM-code of a closed set $D\subset C$ such that $\blambda(C\setminus D)=0$.  
\item The functional $L_{3}$ that to usco $f:[0,1]\di \R$ outputs $x\in [0,1]$ such that $f(x)=\textup{\textsf{ess-sup}}_{y\in [0,1]}f(y)$.  
\item The functional $L_{4}$ that for usco $f:[0,1]\di \R$, decides whether $f$ is continuous almost everywhere.  
\item The functional $L_{4, a}$ that for any $f:[0,1]\di \R$ and its oscillation function $\osc_{f}:[0,1]\di \R$ decides if $f$ is continuous almost everywhere. 
\item The functional $L_{5}$ that for usco $f:[0,1]\di [0,+\infty)$, decides whether $f$ is Riemann integrable.  
\item The functional $L_{5, c}$ that for usco $f:[0,1]\di \R$ decides if the Lebesgue integral $\int_{[0,1]}f$ is zero. 
\item The functional $L_{6}$ that for usco $f:[0,1]\di [0,+\infty)$, decides whether $f$ is finite almost everywhere.  
\item The functional $L_{7}$ that for usco $f:[0,1]\di [0,+\infty)$, decides whether $f$ is differentiable almost everywhere. 
\item The functional $J:(\R\di \R)\di \{0,1\}$ that for closed $C\subset [0,1]$, decides whether $C$ is Jordan measurable.  
\item The function $L_{8}$ which is the Urysohn functional weakened to continuous \emph{almost everywhere} outputs \(or: quasi-continuity ae\). 
\end{itemize}
\end{thm}
\begin{proof}
First of all, the equivalence between $L_{0}$ and $L_{0, c}$ is readily proved as closed sets have usco characteristic functions while $E_{k}:= \{ x\in [0,1]: f(x)\geq \frac{1}{2^{k}}\}$ is closed for usco $f:[0,1]\di \R$.  
Similarly, the first item obviously computes $L_{0}$ and $L_{0, c}$. 
Now assume the Lebesgue measure as in $(\blambda_{C})$ and let $f:[0,1]\di \R$ be usco.  Then $E_{q}:=  \{ x\in [0,1]:  f(x)\geq q \}$ is closed and note that for $n_{0}:=(\mu n)(\blambda (E_{n})=0)$, the essential supremum is in $[n_{0}-1, n_{0}]$.
To compute the first bit (after the comma) of the essential supremum, check whether $\blambda (E_{n_{0}-\frac12})=0$.  Continuing in this fashion, we obtain the binary expansion for the essential supremum, i.e.\ $L_{1}$ can be computed from the Lebesgue measure as in $(\blambda_{C})$.  Clearly, $L_{3}$ is also computable from the Lebesgue measure (and $\exists^{2}$) via the usual interval halving method (using the essential supremum).

\smallskip

Now assume $L_{1}$ and fix closed $C\subset [0,1]$. Note that $\mathbb{1}_{C}$ is usco and split $[0,1]$ into its halves $[0,\frac 12]$ and $[\frac12, 1]$ and set $l_{0}:=1$;  define $l_{1}$ as $l_{0}$ minus $\frac12$ for each sub-interval where the essential supremum is $0$.  Now repeat the previous steps and note that $\lim_{n\di \infty }l_{n}=\blambda(C)$, where the limit is computable by $\exists^{2}$. 
The previous construction also shows that $L_{2}$ is computable from $L_{1}$ and that $L_{0}$ computes $\blambda$.   That $L_{2}$ computes $\blambda$ follows from the fact that $\exists^{2}$ computes the Lebesgue measure 
for RM-closed sets.  

\smallskip

Next, $L_{4}$ is easily computed from $\blambda$ using the decomposition $D_{f}:= \cup_{q\in \Q^{+}}D_{f, q}$ as in \eqref{tago}.  
Indeed, the Lebesgue measure plus $\exists^{2}$ can decide whether each closed set $D_{f, q}$ has measure zero.
For the other direction, compute $\blambda(C)$ using $L_{4}$ as in the second paragraph but remove the length of the interval $[a, b]$ in case $\mathbb{1}_{C}$ is continuous ae on $[a,b]$ and $(\forall q\in \Q\cap [a,b])(\mathbb{1}_{C}(q)=0)$.  Since Riemann integrability on an interval is equivalent to continuity ae and boundedness, the equivalence for $L_{5}$ also follows in this way. 

\smallskip

To compute $L_{4, a}$, use $L_{0, c}$ and note that, by definition, $f$ is continuous ae if and only if the usco function $\osc_{f}$ is zero ae. 
For the other direction, assume $L_{4, a}$ and fix a closed set $C\subset [0,1]$.  Note that `$C$ has non-empty interior' is equivalent to 
\be\label{flagg}\textstyle
(\exists q\in \Q\cap [0,1], N\in \N)(\forall r\in B(q, \frac{1}{2^{N}})\cap \Q)(r\in C).
\ee
Clearly, \eqref{flagg} is decidable using $\exists^{2}$ and in case the former holds, we have $\blambda(C)>0$.  
In case \eqref{flagg} is false, one readily verifies that $\mathbb{1}_{C}$ has itself as oscillation function, i.e.\ $\osc_{\mathbb{1}_{C}}(x)=\mathbb{1}_{C}(x)$ for all $x\in [0,1]$.
We can now apply $L_{4, a}$ to $\mathbb{1}_{C}$ and compute $\blambda(C)$ as in the previous paragraph.  Thus, we can compute $L_{0}$, as required. 

\smallskip

For $L_6$ and $L_{7}$, note that for an usco function $f:[0,1]\di  [0, +\infty)$, we have the following equivalence by \cite{broeker}*{Ch.\ 8, Theorem 1.3, p.\ 113}:
\[
\textup{$f$ is finite ae} \asa \textup{$f$ is differentiable ae}\asa (\lim_{n\di \infty}\blambda(E_{n})=0),
\]
where $E_{q}:=  \{ x\in [0,1]:  f(x)\geq q \}$ is closed.  Hence, $\blambda$ can perform the decision procedure needed for $L_{6}$ and $L_{7}$.
For the other direction, let $C\subset [0,1]$ be closed and define the usco function $\mathbb{2}_{C}:[0,1]\di [0, +\infty)$ 
as follows: $\mathbb{2}_{C}(x):= +\infty$ in case $x\in C$ and $\mathbb{2}_{C}(x)=0$ otherwise.  Now compute $\blambda(C)$ as in the second paragraph but remove the length of the interval $[a, b]$ in case $\mathbb{2}_{C}$ is finite ae on $[a,b]$ (resp.\ differentiable ae).  

\smallskip

Next, to compute $J$ from $\blambda$, fix a closed $C\subset [0,1]$ and note that the boundary $\partial C$ is (always) closed.
To show that $\partial C$ can be defined using $\exists^{2}$ (only), we recall the definition of $\partial E$:
\be\label{gen}\textstyle
x\in \partial E \asa (\forall k\in \N)(\exists y, z \in B(x,\frac{1}{2^{k}} ))(y\in E \wedge z\not \in E).  
\ee
Since $C$ is closed, the previous implies that $x\in \partial C\di x\in C$ for any $x\in [0,1]$.  
Thus, \eqref{gen} reduces to the following:
\[\textstyle
x\in \partial C \asa (\forall k\in \N)(\exists  z \in B(x,\frac{1}{2^{k}} )\cap\Q)( z\not \in C), 
\]
which is arithmetical as the complement of $C$ is open.   Hence, $\partial C$ is closed and definable in $\exists^{2}$, i.e.\ we may use $\blambda(\partial C)=0$ to decide whether $C$ is Jordan measurable, yielding the Jordan functional $J$.  To show that $J$ computes $L_{0}$, let $C\subset [0,1] $ be closed and note that $C$ is the union of $\partial C$ and the interior $\INT(C)$.
Hence, if $J(C)$ tells us $C$ is not Jordan measurable, then $0<\blambda (\partial C)<\blambda (C)$.  
In case $C$ is Jordan measurable, $\blambda(\partial C)=0$ by definition.  Since $C$ is closed, we now have 
\[\textstyle
\blambda (C)>0\asa \INT(C)\ne \emptyset \asa (\exists q\in \Q, k\in \N)(\forall r\in B(x,\frac{1}{2^{k}} ) \cap \Q)(r\in C ),
\]
where the right-most formula is decidable using $\exists^{2}$, i.e.\ $L_{0}$ is obtained. 

\smallskip

To show that $L_{8}$ computes Lebesgue measure for some closed set $C\subset [0,1]$, define $C_{1}:= C$ and $C_{0}:=\emptyset$ and let $f:[0,1]$ be the associated continuous ae function as provided by $L_{8}$.  
Now proceed as in the second paragraph, but omit the interval $[a, b]$ in case $(\forall q\in [a, b]\cap \Q)( f(q)<1 )$.  
To show that $L_{8}$ is computable in terms of $L_{2}$, let $C_{0}, C_{1}\subset [0,1]$ be closed and disjoint.  Now let $D_{0}\subset C_{0}$ and $D_{1}\subset C_{1}$ be as provided by $L_{2}$.  
Since $D_{i}$ are RM-codes for closed sets, $\exists^{2}$ computes the continuous function $f:[0,1]\di \R$ from the (effective) Urysohn lemma as in \cite{simpson2}*{II.7.3}.  
Now define $g:[0,1]\di \R$ as follows
\[
g(x):= 
\begin{cases}
0 & \textup{ if $x\in C_{0}\setminus D_{0}$}\\
1 & \textup{ if $x\in C_{1}\setminus D_{1}$}\\
f(x) &\textup{ otherwise}
\end{cases},
\]
and verify that this function is continuous ae and as required for $L_{8}$.
\end{proof}
Regarding $L_{8}$, the Urysohn lemma for `continuity' replaced by weaker conditions has been studied in e.g.\ \cites{szcz, kowalski, malina}.

\subsubsection{On the computational weakness of the Lebesgue measure}\label{leak}
We show that the Lebesgue measure as in $\blambda_{C}$ is computationally weak.  
Of course, $\blambda_C$ is \emph{lame} since it is computable in the lame functional $\Omega_C$, which follows by the associated cluster theorems.
In this section, we additionally show that $\blambda_{C}$ does not compute $\Omega_{b}$ assuming $\SS^{2}$ (Corollary \ref{corry}) while a slight generalisation of $\blambda_{C}$ does compute $\exists^{3}$, assuming $\exists^{2}$ (Theorem \ref{begga}).

\smallskip

For the sake of simplicity, we define $\blambda_C$ from the product measure on $2^\N$ and not from the Lebesgue-measure on $[0,1]$. Since the binary expansion projection from $2^\N$ to $[0,1]$ is measure-preserving, the two versions of  $\blambda_C$ are equivalent modulo $\exists^2$.

\begin{theorem}
The functional $\Omega_b$ is not computable in $\blambda_C$ and $\exists^2$.
\end{theorem}
\begin{proof}
We let $\omega_1^{\CK}[\vec f]$ denote the least ordinal that is not computable in the sequence $\vec f$ in $\N^\N$.  By the relativised version of the Sacks-Tanaka theorem (\cite{Sacks.high}*{Cor.\ IV.1.6}):
\begin{center} (*)\hspace{2mm} 
If $\omega_1^{\CK}[\vec f] = \omega_1^{\CK}$, then $\{x \in 2^\N : \omega_1^{\CK}[\vec f,x] = \omega_1^{\CK}\}$ has measure 1.
\end{center}
Let $x_{0} \in 2^\N$ be such that $\omega_1^{\CK}[x_{0}] > \omega_1^{\CK}$ and let $X = \{x_{0}\}$. We will use (*) to show that $\blambda_C$ cannot distinguish $X$ from the empty set by some Kleene computation, working with the lambda calculus approach to Kleene computations relative to partial functionals from \cite{dagsamXIII}. Since $\Omega_b$ can distinguish $X$ and $\emptyset$, the theorem follows.

\smallskip

\noindent
In (more) detail, we will prove the following claim:
\begin{itemize}
\item[(**)] Let $Y$ stand  for $X$ or $\emptyset$. Let $\vec f$ be a set of parameters of types 0 or 1 such that $\omega_1^{\CK}[\vec f] = \omega_1^{\CK}$ and let $d[\blambda_C,Y,\exists^2,\vec f]$ be a term of type 0, with all constants displayed, such that the term has a value for both instances of $Y$. Then the two values are the same.

\end{itemize}
The proof will proceed by induction on the rank of the computation tree for $d[\blambda_C,\emptyset,\exists^2,\vec f]$, with a case distinction based on the syntactic form of $d$.
\begin{itemize}
\item In all cases except application of $Y$ or application of $\blambda_C$, all immediate subcomputations will, by the induction hypothesis, have the same values for $Y = \emptyset$ and for $Y = X$, and thus, the value of $d$ itself will be the same in the two cases.
\item In case $d[\blambda_C,Y,\exists^2,\vec f] = Y(\lambda n.e[\blambda_C,Y,\exists^2,n,\vec f])$, the induction hypothesis implies that $e[\blambda_C,\emptyset,\exists^2,n,\vec f]=e[\blambda_C,X,\exists^2,n,\vec f]$ for all $n^{0}$. Now define $y(n) := e[\blambda_C,\emptyset,\exists^2,n,\vec f]$. Since $\blambda_C$ is lame, $y$ is hyperarithmetical in $\vec f$, so $y \neq_{1} x_{0}$ by the assumption on $x_{0}$ and $\vec f$. 
Thus, the value of $d$ for $Y=\emptyset, X$ will be 0, since $y$ is neither in $\emptyset$ nor in $X$.
\item For the case $d[\blambda_C,Y,\exists^2,n,\vec f] = \blambda_C(\lambda y \in 2^\N.e[\blambda_C , Y , \exists^2,y,\vec f])(n)$, we assume that the latter measure is given as a fast converging sequence of rationals. If $y$ is such that $\omega_1^{\CK}[\vec f , y] = \omega_1^{\CK}$, then we use the induction hypothesis and see that the value of $e$ is independent of the choice for $Y$. By (*) above, the remaining set of $y$'s have measure zero, so in evaluating the two values of $d$ we apply $\blambda_C$ to sets of the same measure, and thus the values will be the same.
\end{itemize}
This ends the proof of the claim and the theorem follows.
\end{proof}
By the same argument one can establish the following corollary.
\begin{cor}\label{corry}
The functional $\Omega_b$ is not computable in $\blambda_C+\SS^{2}$.
\end{cor}
\begin{proof}
We apply the relativised version of \cite{nomu}*{Theorem 3.9}.
We cannot relativise the lameness of $\Omega_C$ to $\SS^{2}$, but by the results in \cite{nomu} we can relativise the lameness of $\blambda_C$ to $\SS^{2}$, so the full proof works.
\end{proof}
Finally, we have restricted $\blambda_{C}$ to closed sets and it is a natural question what happens if this restriction is lifted.  
To this end, let the functional $L_{0,a}(E)$ be defined for all subsets of $[0,1]$ and decide if $\partial E$ has measure 0 or not.
Since $\partial E$ is closed for \emph{any} set $E\subset \R$, the functional $L_{0, a}$ seems similar enough to $\blambda_{C}$, which is an incorrect impression by the following result. 
\begin{thm}\label{begga}
The combination $L_{0,a}+\exists^2$ computes $\exists^{3}$.  
\end{thm}
\begin{proof}
Let $X \subseteq [0,1]$ be given. To decide if $X=\emptyset$, proceed as follows: in case $\partial X$ has positive measure, $X$ is non-empty.  
In case $\partial X$ has measure zero, let $A \subset [0,1]$ be a fat Cantor set, i.e.\ $A$ is closed, has positive measure, and has empty interior. 
Using the standard construction of such a set $A$, there is a primitive recursive sequence of pairs $ \langle a_i,b_i\rangle$ of rationals such that $(a_i,b_i)$ are pairwise disjoint and their union form the complement of $A$. We let $0,1 \in A$. Due to the lack of an interior, $A$ will be the closure of $\{a_i,b_i : i \in \N\}$, and all $a_i$ and $b_i$ are in $A$.

\smallskip

In a nutshell, we shall define a set $E\subset [0,1]\setminus A$ as in \eqref{fif} such that: $E$ is empty if $X$ is empty, while $\partial E=A$ if $E\ne \emptyset$.
Applying $L_{0, a}$, we can decide if $X=\emptyset$ and the proof is finished. All items involved in the construction of $E$ are arithmetically defined, and thus computable in $\exists^2$.

\smallskip

Each interval $(a_i,b_i)$ can be put in a 1-1 correspondence with $\R$   via a continuous, increasing function $\phi_i$. For $k \in \mathbb{Z}$ we let $I_{i,k} = \phi_i^{-1}([2k,2k+1])$. 
These intervals are closed intervals of positive length and we define $\psi_{i,k}$ as the affine 1-1-map from $I_{i,k}$ to $[0,1]$.
Finally, we define the following set:
\be\label{fif}
E = \bigcup_{i \in \N,k \in \mathbb{Z}}\psi_{i,k}^{-1}(X).
\ee
Next, if $X=\emptyset$, then $E=\emptyset$ and $\partial E$ has measure zero. If $X\ne \emptyset$, then each $a_i$ and $b_i$ will be in the closure of $E$, and consequently, all of $A$ will be in the closure of $E$. 
The rest of the closure of $E$ is contained in the intervals $I_{i,k}$. 
A real $x\in A$ cannot be an interior point in the closure of $E$, since we left gaps out of $E$ arbitrarily close to the endpoints $a_i$ and $b_i$ for each $i$; for the other elements in the closure of $E$ they cannot be an interior point because the corresponding $\psi_{i,k}(y) \in [0,1]$ is not an interior point in the boundary of $X$.
\end{proof}
The difference between $L_{0}$ from Theorem \ref{lam} and $L_{0,a}$ from Theorem \ref{begga} is of course that while $\partial E$ is closed for any $E\subset \R$, we cannot define the former using $\exists^{2}$.  
In this (subtle) way, $L_{0}$ does not compute $L_{0,a}$.  
\section{Some open problems related to $\Omega_{C}$}\label{hopen}
In this 
section, we discuss some open problems relating to partial sub-functionals of $\exists^3$ that are generalisations of $\Omega_{C}$ and $\Omega$.  

\smallskip

First of all, let $\Gamma$ be a family of subsets of $[0,1]$, alternatively of $\R$, $2^\N$ or $\N^\N$. Then $\Gamma$ may serve as the domain of a partial subfunction $\Omega_\Gamma$ of $\exists^3$ (defined for one of the alternative spaces) deciding if an element $X$ of $\Gamma$ will be empty or not. For $\Omega_\Gamma$ to be non-trivial, we must have that $\Gamma$ contains the empty-set.

\smallskip

Both $\Omega_{\rm b}$ and $\Omega_{\rm C}$ are examples of such partial subfunctions of $\exists^3$, the main examples in this paper. Related to them is one crucial problem that we have to leave open, as follows.
We conjecture that the answer is negative.
\begin{problem}\label{whatsyour}
{\em Is $\Omega_{\rm C}$ computable in $\Omega_{\rm b}$ and $\exists^2$?}
\end{problem}
We have displayed a lot of examples of why $\Omega_{\rm b}$ and $\Omega_{\rm C}$ are of interest to the foundational study of ordinary mathematics, 
but from the point of view of computability theory there may be other functionals $\Omega_\Gamma$ with interesting properties. We will mention a few examples, reveal what we know about them, and leave the rest for future research.
\begin{itemize}
\item $\Gamma_1$ is the set of closed subsets of $\N^\N$, alternatively the set of ${\bf G}_\delta$ subsets of $[0,1]$, of $\R$ or of $2^\N$
\item $\Gamma_2$ is the set of ${\bf F}_\sigma$ subsets of $[0,1]$, of $\R$ or of $2^\N$, alternatively the set of $\sigma$-compact subsets of $\N^\N$.
\item $\Gamma_3$ is the set of countable subsets of the space in question.
\item $\Gamma_4$ is the set of countable, closed subsets of the space in question.
\end{itemize}
Little is known about the relative computability of these functionals (modulo $\exists^2$), beyond\footnote{We now know that $\Omega_{\Gamma_3}$ is not computable in $\Omega_{\rm C}$ and $\exists^2$.} what follows from inclusions between the domains and the observations below.
 It is easy to see that the Suslin functional $\SS^{2}$ is computable in $\Omega_{\Gamma_1}$ and $\exists^2$, and thus that it is consistent with {\bf ZF} that $\exists^3$ is computable in the two. In this way, we observe that $\Omega_{\Gamma_{1}}$ is beyond the scope of this paper.

\smallskip

On the other hand, if a set $X \neq \emptyset$ is in $\Gamma_2$ and hyperarithmetical in a function $f$, then $X$ contains\footnote{This is a consequence of the following corollary of the Louveau separation theorem (\cite{louveau1}): {if an $\textbf{F}_\sigma$-set $X$ is hyperarithmetical (in $f$), then $X$ has a hyperarithmetical code as an $\textbf{F} _\sigma$-set. }} some element hyperarithmetical in $f$. 
This is what is needed to prove that $\Omega_{\Gamma_2}$ is lame, and then $\Omega_{\Gamma_3}$ and $\Omega_{\Gamma_4}$ will be lame as well. In a sense $\Omega_{\Gamma_2}$ is a sleeping monster, needing the Suslin functional to wake it up, while $\Omega_{\Gamma_1}$ is fully awake.  It is a natural question is if it is possible to draw an even finer line between these two kinds of sub-functionals of $\exists^3$.

\section{A lambda calculus formulation for S1-S9}\label{lambdaz}
\subsection{Introduction}
In a nutshell, the below is a corrigendum to the paper
\begin{center}
\emph{On the Computational Properties of Basic Mathematical Notions} (\cite{dagsamXIII})
\end{center}
by the authors.  As discussed above, this section is also relevant to the previous sections as we present an extension of S1-S9-computability to partial objects.  

\smallskip

In more detail, we have formulated a higher-order computability theory in \cite{dagsamXIII}, based on fixed point operators and the lambda calculus. 
Unfortunately, there is a technical error in \cite{dagsamXIII}, which was pointed out by John Longley in a private communication. 
The error in \cite{dagsamXIII} is in Section 3, where the restriction in bullet point nr.\ 5 of Definition 3.7 is too strong. As a consequence, Theorem 3.19 of \cite{dagsamXIII} is not correct, which will all be fixed in the below.

\smallskip

In this corrigendum we provide a correction to the computability theory from \cite{dagsamXIII} based on ideas from Platek, Scott, Plotkin, and others.  We formulate an alternative approach to higher-order computability, relevant for the investigation of the relative computability of realisers. The key concept is that of a \emph{computation tree}, an operational-like way of understanding higher-order computations.
 
\smallskip

Mathematically, there is hardly anything new in this corrigendum.  All type structures we consider, and that are needed in order to rectify the error in \cite{dagsamXIII}, are present in Platek's thesis (\cite{Pla}), as surveyed in Moldestad's \cite{Mol}*{\S12}. The lambda calculus we present in the below may be found in Platek's thesis (\cite{Pla}). The schemes S1-S9 have to be modified in order to deal with functionals of all finite types, but these modifications do not add any extra computational power. The details of those modifications will not be needed, and are therefore omitted.

\smallskip

The equivalence between the Kleene-style notion of computability and that based on the lambda calculus is also due to Platek, and can be found in published form in \cite{Mol}.  The proof is however quite involved, and for expository reasons we employ methods developed by Plotkin \cite{Plo} to give a more transparent treatment.  In particular, we define computation trees directly from the lambda-terms and not via the representation as a computation theory based on schemes. Thus, Theorem \ref{thm5.7} below is original and captures the purpose of our calculus from \cite{dagsamXIII}. 
Following Theorem \ref{thm5.7}, the equivalence with the Kleene-style approach is an easy exercise, and thus left out. 

\smallskip

All applications of our computing framework in \cite{dagsamXIII} are still sound, since they are based on the computation tree characterisation of computable functionals.
All functionals of interest to this paper are in some $\Pa(\sigma)$, defined in Definition~\ref{Dree}. We consider Theorem \ref{thm5.7} -and its transparent proof- to represent a significant simplification of the proofs in \cite{Pla} (and \cite{Mol}*{\S12}) when restricted to  $\Pa$.

\smallskip

The authors are grateful to John Longley for spotting the error in \cite{dagsamXIII} and for discussing intermediate attempts to deal with this error. The authors are also grateful to Johan Moldestad for pointing out how \cite{Mol}*{\S12} is related to \cite{Pla}.
\subsection{The computation theory}\label{s2}
In this section, we introduce our computation theory in full detail.

\smallskip

First of all, the \emph{types} of finite rank are inductively defined as usual.
\begin{definition}[Finite types]~
\begin{itemize}
\item The type of natural numbers, $0$, is a type.
\item If $\tau$ and $\delta$ are types, then $\sigma = (\tau \rightarrow \delta)$ is a type.
\end{itemize}
%
\end{definition}
We adopt the standard notation based on currying. We will consider three intimately related type-structures, as follows.
\begin{definition}[Three type structures]\label{Dree}{\em  We define the structure $(\Mo(\sigma ), \preceq_\sigma)$, the sets $\F(\sigma)$ and the sets $\Pa(\sigma)$ 
by recursion on the rank of $\sigma$ as follows.
\begin{itemize}
\item We define $\Mo(0) = \Pa(0) = \N_\bot = \N \cup\{\bot\}$ and $\F(0) = \N$.  The associated ordering $x \preceq_{0} y$ holds if $x = \bot$ or $x = y$.
\item We define $\Phi \in \Mo(\sigma \rightarrow \tau)$ if $\Phi:\Mo(\sigma) \rightarrow \Mo(\tau)$ and is monotone with respect to $\preceq_\sigma$ and $\preceq_\tau$. The associated ordering $\Phi \preceq_{\sigma \rightarrow \tau} \Psi$ holds if we have monotonicity as in $\phi \preceq_\sigma \psi \rightarrow \Phi(\phi) \preceq_\tau \Psi(\psi)$ for any $\phi^{\sigma}, \psi^{\sigma}\in \Mo(\sigma)$. 
\item We define $\Pa(\sigma \rightarrow \tau) = \F(\sigma) \rightarrow \Pa(\tau)$ and $\F(\sigma \rightarrow \tau) = \F(\sigma) \rightarrow \F(\tau)$.
\end{itemize}
We sometimes omit the typing in $\preceq_{\sigma}$ if this information is clear from the context. 
}\end{definition}
By recursion on the type $\sigma$, we define an \emph{embedding} $\Phi \mapsto \varepsilon_\sigma(\Phi)$, mapping each $\F(\sigma)$ and $\Pa(\sigma)$ into $\Mo(\sigma)$, thereby identifying the  spaces $\F(\sigma)$ and $\Pa(\sigma)$ with subsets of $\Mo(\sigma)$, as in the following definition.
\begin{definition}[Embedding]
{\em \hspace*{2mm}
\begin{itemize}
\item The mapping $\varepsilon_0$ is the identity map both on $\F(0)$ and on $\Pa(0)$.
\item If $F \in \Pa(\sigma \rightarrow \tau)$, $\phi \in \Mo(\sigma)$ and there is a $g \in \F( \sigma )$ such that $\varepsilon_\sigma(g) \preceq \phi$, then $\varepsilon_{\sigma \rightarrow \tau}(F)(\phi) = \varepsilon_{\tau}(F(g))$. If there is no such $g$, $\varepsilon_{\sigma \rightarrow \tau}(F)(\phi) := \bot_\tau$.
\end{itemize}
}\end{definition}
%

We will  now introduce the term language $\mathcal L$ to be used by our computability theory. This language is based on Plotkin's PCF (\cite{Plo}).
\begin{definition}[The language $\mathcal L$]\label{definition8}{\em ~
\begin{itemize}
\item The constants of our language $\mathcal L$ are as follows: $ 0$ of type 0,
 ${\suc}$ of type $0 \rightarrow 0$,
${\pd}$ of type $0 \rightarrow 0$,
${\case}$ of type $(0\times 0\times 0) \rightarrow 0$ and $\Phi$ of type $\sigma$ for each $\Phi \in \Pa(\sigma)$, the latter interpreted as  $\varepsilon_\sigma(\Phi)$.
\item  The typed terms in the language $\mathcal L$ are defined by recursion as follows.
\begin{itemize}
\item All  constants are terms in $\mathcal L$.
\item For all types $\sigma $, there is an infinite list $x_i^\sigma$ of variables of type $\sigma$. These are terms in $\mathcal L$.

\item If $t$ is a term of type $\sigma \rightarrow \tau$ and $s$ is a term of type $\sigma$, then $(ts)$ is a term of type $\tau$. 
\item If $t$ is a term of type $\tau$, then $(\lambda x_i^\sigma .t)$ is a term of type $\sigma \rightarrow \tau$. 
\item If  $t$ is a term of type $\sigma$, then $(\Fix~x_i^\sigma .t)$ is a term of type $\sigma$.
\end{itemize}
\end{itemize}}
\end{definition}
Following the previous definition, any term $t$ of type  $\sigma$ with free variables $x_1^{\tau_1}, \ldots x_m^{\tau_m}$ has a canonical interpretation as a monotone function 
\[ 
[[t]]:\Mo(\tau_1) \times \cdots \times \Mo(\tau_m) \rightarrow \Mo(\sigma),
\]
where we use the existence of least fixed points of monotone functions from $\Mo(\sigma)$ to $\Mo(\sigma)$ to interpret the fixed point operator $\Fix~x.s$. 

\smallskip

We are now ready to introduce our notion of computability.
\begin{definition}\label{def.comp}{\em 
Let $\sigma_1 , \ldots , \sigma_n , \delta$ be types, and let $\phi_1 , \ldots , \phi_n,\psi$ be partial objects of the corresponding types. 
We say that $\psi$ \emph{is computable in} $\phi_1 , \ldots , \phi_n$ if there is a term $s$ of type $\tau$,  with free variables among $x_1 , \ldots , x_n$ of types $\sigma_1 , \ldots , \sigma_n$ such that 
\[
  \psi \preceq_\tau [[s]](\phi_1 , \ldots , \phi_n ) .
  \]
}\end{definition}
The validity of $\beta$-conversion for our denotational semantics is not hard to prove, and is a consequence of the following lemma.
\begin{lemma} Let $t$ be a term with free variables among $y$ of type $\tau$ and $\vec x$ and let $s$ be a term of type $\tau$ with free variables among $\vec x$. Let $\vec \phi$ be a sequence of objects of the types of $\vec x$. Then

\[ [[t^y_s]](\vec \phi) = [[t]]([[s]](\vec \phi), \vec \phi).\]

\end{lemma} 
The proof is by induction on the syntactic form of $t$, and all cases are trivial. In particular, the initial case $t = y$ is a tautology. In the case of the $\Fix$ operator, a sub-induction on the ordinals is required.

\subsection{An operational-like semantics}\label{s4}
In this section, we define a semantics for our notion of computability.  We always assume that $t$ is a closed term of type 0.  

\smallskip

In particular, we define the notion of \emph{computation tree} $T[t]$ which represents an operational evaluation of $[[t]]$. Our definition is  inspired by the operational semantics for PCF as designed by Plotkin in 1977, but since we are modelling infinitary computations, our modifications have to be far-reaching. The main motivation is that we want to recapture the qualities of the computation trees in Kleene's  computability model for calculations with inputs from the $\P$-structure as in Definition~\ref{Dree}. We can then prove the equivalence between the denotational semantics and our operational one.
We use the fact that all terms of type $0$ can be written as an iterated composition $t = t_1 \cdots t_n$ where $t_1$ is not itself a composition.

\begin{definition}\label{definition12}{\em
Let $t$ be a closed term of type $0$ with constants $\Phi_1, \ldots ,  \Phi_n$. We define the  \emph{computation tree} $T[t]$ of $t$, formally consisting of sequences of terms, by considering a closed term of type $0$ as a \emph{computation} and  defining what the immediate sub-computations are. Some terms $t$ will have a value $a \in \N$, while others are undefined. The definition of $T[t]$ is top-down, while the definition of the value is bottom-up. The whole construction can be viewed as a positive inductive definition, in the intended analogy of the set of Kleene computations. 
\begin{enumerate}
\item If $s= a$, where $a$ of type 0 is a constant for itself, then $t$ is a leaf node with $a$ as the value.  The cases $s = \suc\; s_1$, 
 $s = \pd \;s_1$, and $s= \case \;s_1s_2s_3$, are left for the reader.
\item If $s = \Phi s_1 \cdots s_n$ with $n \geq 1$, then $\Phi \in \Pa(\sigma \rightarrow \tau)$ for some $\sigma$ and $\tau$ and  $s_1$ will be of  type $\sigma = \delta_1 , \ldots , \delta_m \rightarrow 0$ for some $\delta_1 , \ldots , \delta_m$.
 \begin{itemize}
  \item[-] For each $(\phi_1 , \ldots , \phi_m) \in \F(\delta_1) \times \cdots \times \F(\delta_m)$ we  let $ s_1\phi_1 \cdots \phi_m$ be an immediate\footnote{Note that if $m = 0$, then $s_1$ is of type 0 and is itself a predecessor. This case is implicitly treated in the general treatment below.
} predecessor.  
  \item[-] Then, if each such  predecessor has a value, these  values will define  a functional $\Psi \in \F(\sigma)$. Then $\xi = \Phi(\Psi) \in \Pa(\tau)$. In the case $\tau \neq 0$, i.e. $n \geq 2$,  the term $\xi s_2 \cdots s_n$ is also a predecessor, and in case it has a value, this will be the value  of $s$.  In the case $\tau = 0$, then $\xi \in \N_\bot$. The value of $s$ will be $\xi$ if $\xi \in \N$, otherwise there will be no value. This will be a leaf in the computation tree.
 \end{itemize}
 \item If $s =( \lambda y .s_1)s_2 \cdots s_n$ we let $(s_1)^y_{s_2}s_3 \cdots s_n$ be the predecessor, and the value will be preserved. %
 \item $s = (\Fix~x .s_1)s_2 \cdots s_n$. Then $ (s_1)^x_{(\Fix~x .s_1)}s_2 \cdots  s_n$ is the only predecessor, and the value will be preserved.

\end{enumerate}
}\end{definition}
It is clear that for a computation tree $T[t]$ to have a value, it must be well-founded, since a value always depends on the values of the predecessors, except for the leaf terms of type $0$, where the value is just given. The converse is not always true, since an application involving partial functionals may yield a well-founded tree without a value.
\bigskip

\begin{theorem}\label{thm5.7} Let $t$ be a closed term of type 0 in $\mathcal L$. Then  the following are equivalent
\begin{enumerate}
\item The interpretation of $t$ satisfies $[[t]] = a$ for some $a \in \N$.
\item  The computation tree $T[t]$ is well-founded with value $a$.
\end{enumerate} \end{theorem}
\begin{proof}
First of all, the proof of $(2) \rightarrow (1)$ is by induction on the ordinal rank of $T[t]$ and organised in cases according to the syntax of the leading term in $t$. We omit the details, just observing that we use the validity of $\beta$-conversion to handle the case where we have $t = \lambda x.s_1 \cdots s_n$.

\smallskip

Secondly, to prove $(1) \rightarrow (2)$ we adapt a proof given in Plotkin \cite{Plo}. In particular, for each ordinal $\alpha$ we introduce the auxiliary terms $\Fixa_{\alpha}~x_i.t$ with the denotational semantics (for the special case $i = 1$) as follows:
\[ 
[[\Fixa_\alpha~x_1.t \xi_1 \cdots \xi_m]](\phi_2, \ldots , \phi_n) = \bigsqcup_{\beta  < \alpha}[[ t^{x_1}_{\Fixa_\beta~x_i1t}\xi_1 , \cdots \xi_m]](\phi_2, \ldots , \phi_n).
\]
We then have for all terms $t$ that
\[ 
[[\Fix~x.t]] = \sqcup_{\alpha \in {\rm On}}[[\Fixa_\alpha~x.t]].
\] 
As a consequence, there will be a cardinal $\kappa$ such that $\Fix$ and $\Fixa_\kappa$ will have the same denotational interpretation.

\smallskip

Now, for each ordinal $\alpha \leq \kappa$, let ${\mathcal L}_\alpha$ be $\mathcal L$, but where we use  $\Fixa_\beta$ for $\beta \leq \alpha$ instead of $\Fix$. 
If $t$ is a term in ${\mathcal L}_\alpha$ for some $\alpha$, then let $t^-$ be the term where all occurrences in $t$ of $\Fixa_\beta$ for some $\beta$ are replaced by $\Fixa$.
We now define the notion of \emph{normality} for all terms in ${\mathcal L}_\kappa$.
\begin{itemize}
\item[(i)] If $t$ is closed and of type 0, then $t$ is \emph{normal} if $[[t]] = \bot$ or $T[t^-]$ is well-founded and with $[[t]]$ as value.
\item[(ii)] If $t$ is closed and of type $\sigma \rightarrow \tau$, then $t$ is normal if $ts$ is normal whenever $s$ is closed and normal.
\item[(iii)] If $t$ has free variables $x_1 , \ldots x_n$, then $t$ is normal if $t^{x_1 , \ldots , x_n}_{s_1 , \ldots, s_n}$ is normal whenever $s_1 , \ldots , s_n$ are closed and normal.
\end{itemize}
We implicitly assume that all typings are correct. Our key claim is now as follows.
\begin{center}  
\emph{All terms in ${\mathcal L}_\kappa$ are normal.}
\end{center}
In order to prove this claim, we use induction on the syntax of $t^-$, and in the cases of $\Fixa_\alpha~x.s$, subinduction on $\alpha$. It suffices to prove for all closed terms $t$ of type $\sigma = \tau_1 , \ldots , \tau_m \rightarrow 0$ that:'
\begin{center}
\emph{if each $s_i$ is closed, normal, and of type $\tau_i$ for $i = 1 , \ldots , m$ and $[[ts_1\cdots s_m]] = a$, then the computation tree of  $ts_1\cdots s_m$ is well-founded with value $a$.}
\end{center}
All cases are easy and most are well-known. In the case when $t$ is a constant $\Phi$ from $\Pa$ we need a sub-induction on the type rank of $\Phi$. We leave the (straightforward) details for the reader.\end{proof}

\begin{ack}\rm 
We thank John Longley for spotting the error in the lambda calculus from \cite{dagsamXIII}.
We thank the anonymous referees for their constructive suggestions. 
We thank Anil Nerode for his valuable advice, especially discussions related to Baire classes.
We thank Dave L.\ Renfro for his efforts in providing a most encyclopedic summary of analysis, to be found online.  
Our research was supported by the \emph{Klaus Tschira Boost Fund} via the grant Projekt KT43.
We express our gratitude towards the latter institution.    
\end{ack}

\bye